\documentclass{amsart}
\usepackage{MnSymbol}

\newbox\gnBoxA
\newdimen\gnCornerHgt
\setbox\gnBoxA=\hbox{$\ulcorner$}
\global\gnCornerHgt=\ht\gnBoxA
\newdimen\gnArgHgt
\def\Quinequote #1{%
    \setbox\gnBoxA=\hbox{$#1$}%
    \gnArgHgt=\ht\gnBoxA%
    \ifnum     \gnArgHgt<\gnCornerHgt \gnArgHgt=0pt%
    \else \advance \gnArgHgt by -\gnCornerHgt%
    \fi \raise\gnArgHgt\hbox{$\ulcorner$} \box\gnBoxA %
    \raise\gnArgHgt\hbox{$\urcorner$}}

\newcommand{\Q}{\mathbb{Q}}
\newcommand{\C}{\mathbb{C}}

\newcommand{\R}{\mathbb{R}}
\newcommand{\N}{\mathbb{N}}

\newcommand{\dom}{\operatorname{dom}}
\newcommand{\ran}{\operatorname{ran}}

\newcommand{\norm}[1]{\left\| #1 \right\|}

\newcommand{\B}{\mathcal{B}}
\newcommand{\zerovec}{\mathbf{0}}

\newcommand{\F}{\mathbb{F}}
\newcommand{\Rat}{\mathcal{R}}
\newcommand{\IncludesClosure}{\operatorname{IC}}

\newcommand{\Banach}{\operatorname{Banach}}

\newcommand{\VectorTree}{\operatorname{VectorTree}}

\newcommand{\Disint}{\operatorname{Disint}}

\newcommand{\Hilbert}{\operatorname{Hilbert}}
\newcommand{\Lspace}{\operatorname{Lspace}}
\newcommand{\Lpres}{\mathcal{L}_{\operatorname{pres}}}
\newcommand{\Ldisint}{\mathcal{L}_{\operatorname{disint}}}
\newcommand{\Lchain}{\mathcal{L}_{\operatorname{chain}}}
\newcommand{\Phidisint}{\Phi_{\operatorname{Disint}}}
\newcommand{\Phidis}{\Phi_{\operatorname{Disint}}}
\newcommand{\Phidecomp}{\Phi_{\operatorname{Decomp}}}
\newcommand{\Ind}{\mathbf{1}}

\theoremstyle{theorem}
\newtheorem{theorem}{Theorem}[section]
\newtheorem{lemma}[theorem]{Lemma}

\newtheorem{maintheorem}{Main Theorem}

\theoremstyle{definition}
\newtheorem{definition}[theorem]{Definition}

\theoremstyle{theorem}
\newtheorem{corollary}[theorem]{Corollary}

\theoremstyle{theorem}
\newtheorem{proposition}[theorem]{Proposition}

\theoremstyle{theorem}

\theoremstyle{theorem}

\theoremstyle{definition}

\theoremstyle{theorem}

\numberwithin{equation}{section}

\begin{document}
\title{On the complexity of classifying Lebesgue spaces}

\author{Tyler A. Brown}
\address{Department of Mathematics\\
Iowa State University\\
Ames, Iowa 50011 USA}
\email{tab5357@iastate.edu}

\author{Timothy H. McNicholl}
\address{Department of Mathematics\\
Iowa State University\\
Ames, Iowa 50011 USA}
\email{mcnichol@iastate.edu}

\author{Alexander G. Melnikov}
\address{The Institute of Natural and Mathematical Sciences\\
Private Bag 102 904 NSMC \\
Albany 0745 Auckland, New Zealand
}
\email{alexander.g.melnikov@gmail.com}
\thanks{Part of this research was conducted while T. McNicholl visited A. Melnikov.  This visit was funded by Marsden Fund of New Zealand and by Simons Foundation Grant \# 317870.}

\begin{abstract}
Computability theory is used to evaluate the complexity of classifying various kinds of Lebesgue spaces and associated isometric isomorphism problems.
\end{abstract}
\maketitle

\section{Introduction}\label{sec:intro}
This paper advances and interleaves two general frameworks.  The first framework, which was proposed in \cite{MelIso}, is focused on establishing technical connections between computable structure theory~\cite{AshKn, ErGon}  and computable analysis~\cite{PourElRich, Wei00}. Recently, there have been a number of applications of computable algebraic techniques to the study of effective processes in Banach and metric spaces; see, e.g.,  \cite{MelNg, CompComp, McNicholl.Stull.2019, Clanin.McNicholl.Stull.2019, MeMo}.  The second framework focuses on applying computability-theoretic techniques to classification problems in mathematics. Computable structure theory~\cite{AshKn,ErGon} provides tools for understanding the complexity of
classification and characterisation problems for various classes of algebraic structures (details later).   See \cite{CHKLMMQSW12,McCoyWallbaum12,GonKni, friedberg, DoMo} for further recent applications of computability to classification problems.

Herein, we apply an approach borrowed from effective algebra to produce
a fine-grained algorithmic characterization of separable Lebesgue spaces among all separable Banach spaces. (Recall that an $L^p$ space is a space of the form $L^p(\Omega)$ where $\Omega$ is a measure space, and a Lebesgue space is a space that is an $L^p$ space for some $p$.).  We also measure the complexity of the isometric isomorphism problem (to be defined) for separable Lebesgue spaces; specifics below.

\subsection{Index sets in discrete algebra} Goncharov and Knight \cite{GonKni} suggested a number of applications of computable structure theory to classification problems. We adopt the most common approach via \emph{index sets}.  Recall that a countable structure is \emph{computable} if its domain is the set of natural numbers and if its operations and relations are uniformly Turing computable~\cite{ma61,abR}.  An \emph{index} of a computable structure is an index of a Turing machine that computes these operations and relations.  An index of a structure may be regarded as a finite description of the structure.  
 
  Fix some property $P$; for example, $P$ could be ``is a directly decomposable abelian group".  
  The \emph{index set} of $P$ is the set of all natural numbers that index a structure with property $P$; denote this set by $I_P$.  
  The complexity of the property $P$ is reflected in the complexity of $I_P$ which 
 is usually measured using various hierarchies such as the arithmetical and the analytical hierarchies~\cite{rogers,Soa}.
 The classes in these hierarchies correspond to the number (and type) of quantifiers required to solve the problem. For example, if $I_P$ is $\Sigma^1_1$-complete, then this means that solving the problem requires searching through the uncountably many elements of Baire space $\omega^{\omega}$ in a brute-force fashion. In contrast, if $I_P$ is in either $\Sigma^0_n$ or $\Pi^0_n$, then the decision procedure for $P$ requires understanding of merely $n$-$1$ alternations of quantifiers over natural numbers. It is not difficult to show that all these hierarchies are  proper; see~\cite{rogers,Soa}. For instance, for every $n$ both complexity classes $\Sigma^0_n$ or $\Pi^0_n$ are properly contained in $\Sigma^1_1$. 
 
 For example, using algorithmic tools  Riggs~\cite{riggs} showed that  decomposability of an (abelian) group is a  $\Sigma^1_1$-complete problem. The result of Riggs means that the there is no reasonable way of characterizing non-trivially directly decomposable (abelian) groups.  This is because any reasonable necessary and sufficient condition would make the decomposability property simpler than the brute-force upper bound $\Sigma^1_1$. In stark contrast, complete decomposability of an abelian group is merely $\Sigma^0_7$~\cite{DoMel1}. Usually such results can be  relativized to any oracle. For instance, the above-mentioned results of Riggs and Downey and Melnikov work for arbitrary discrete countable abelian groups.

To measure the complexity of isomorphism,  consider pairs of indices of isomorphic structures in a class. 
  For instance,  the classification of vector spaces  by dimension allows to show that \emph{the isomorphism problem}
 $$\{ \langle i, j \rangle \in \omega \mid A_i, A_j \text{ are vector spaces over $F$ and } A_i \cong A_j\}$$
 is merely $\Pi^0_3$-complete, where  $\langle i, j \rangle   = 2^i 3^j$. 
In contrast, Downey and Montalb\'an~\cite{DoMo} proved that the isomorphism problem for torsion-free abelian groups os $\Sigma^1_1$-complete.
Consequently,
there is no better way to check if two (countable, discrete) torsion-free abelian groups are isomorphic than to search through the uncountably many potential isomorphisms. From the perspective of computability theory, it follows that such groups are unclassifiable up to isomorphism. The abundance of ``monstrous'' examples of such groups in the literature~\cite{Fu,Fu1} strongly support
this conclusion.
 Compare this to vector spaces, free groups, abelian $p$-groups of bounded type or completely decomposable groups which do possess convenient invariants~\cite{DoMo, DoMel1, McCoyWallbaum12, CalHarKnMi2006}. We again emphasize that all these results can be fully relativized and therefore are not restricted to computable members in the respective class.  All results that we mentioned so far are concerned with countable discrete algebraic structures; furthermore, it seems that potential applications of index sets are naturally limited to countable objects.
Nonetheless, a similar methodology has recently been applied to study 
 the complexity of finding a basis in a (discrete) uncountable free abelian group~\cite{GrTurWeFree}. In this paper we will also apply index sets to uncountable objects, but our \emph{analytic} approach is rather different from the set-theoretic one taken in~\cite{GrTurWeFree}. 
 
  \subsection{Index sets in computable analysis} The use of index sets in computable analysis is not entirely new. In the late 1990s, Cenzer and Remmel~\cite{CRindex} used index sets to measure the complexity of effectively closed subsets of standard elementary metric spaces such as $2^{\omega}$. However, index sets in computable analysis have only recently been linked to classification problems for separable spaces. 
  
  The idea here is that, similarly to discrete computable algebras, one can define the notion of a computable presentation of a separable metric space~\cite{Wei00}. 
For example, Turing~\cite{Turing:36, Turing:37}  used density of the rationals to define computable real numbers. Thus, the standard computable copy of the rationals can be viewed as a computable presentation of $\mathbb{R}$. Similarly, we follow \cite{Wei00} and say that a \emph{computable presentation} of -- or  a \emph{computable structure}  on --  a separable metric space is any dense computable sequence $\{p_n\}_{n \in \N}$ of points in the space such that the metric is uniformly computable for points in the sequence.  That is, there is a Turing machine that given $m,n,k \in \N$ produces a rational number $q$ so that $|q - d(p_m,p_n)| < 2^{-k}$.  An index of such a Turing machine is referred to as an \emph{index} of the presentation. 

We emphasize that a presentation of a separable space $M$  does not have to be ``standard". For example, fix any (not necessarily computable) real $\xi$ and consider the collection $(r +\xi)_{r \in \mathbb{Q}}$. Then $(r +\xi)_{r \in \mathbb{Q}}$  is a computable structure on the reals equipped with the usual distance metric $d_{e}(x,y) = |x-y|$. Although $(r +\xi)_{r \in \mathbb{Q}}$ is not \emph{equal} to the ``standard'' structure  $(r)_{r \in \mathbb{Q}}$ on $(\mathbb{R}, d_{e})$, it is easy to see that $(r +\xi)_{r \in \mathbb{Q}}$ and  $(r)_{r \in \mathbb{Q}}$ are \emph{computably isometric} \cite{MelIso}. More generally, any two computable structures, ``natural'' or not, on $\mathbb{R}$ are computably isomorphic, and the same can be said about any separable Hilbert space~\cite{MelIso}.
Note however that many standard metric spaces associated with Banach spaces possess computable structures which are not computably isometric; examples include $(C[0,1], sup)$ and $(\ell^1, d_{1})$; see \cite{MelIso,MelNg} for more examples. 

Now when we have the notion of a computable structure on a Polish metric space and fixed the right morphisms for this category, we can act by analogy with  computable structure theory and list all presentations of separable spaces and study various index sets. 
  Using this novel approach, Melnikov and Nies~\cite{CompComp} showed that  the isomorphism problem and the index set of compact metric spaces are both arithmetical. In contrast, Nies and Solecki~\cite{NSlocal} showed that the index set of locally compact spaces is $\Pi^1_1$-complete, and thus there is no reasonable characterisation of such spaces which would be simpler than the brute-force definition.   Melnikov~\cite{Pontr} used Pontryagin duality theory to illustrate that the (topological) isomorphism problems for compact connected and profinite abelian groups are both $\Sigma^1_1$-complete, and therefore such groups cannot be classified by reasonable invariants; compare this with the above-mentioned results of Melnikov and Nies~\cite{CompComp}.  All of these results relativize.
  
Similarly to computable Polish metric spaces, computable Banach spaces also admit computable presentations.  These presentations are formally defined in Section \ref{sec:back::subsec:CA}.  In brief, a presentation of a Banach space consists of a linearly dense sequence $\{v_n\}_{n \in \N}$ of vectors.  A presentation of a Banach space is computable if the norm function is computable on the set of rational linear combinations of $v_0,v_1, \ldots$.  An index of a Turing machine that effects this computation is then referred to as an index of the presentation.  Again, an index of a Banach space presentation may be viewed as a finite description of the space. 
This approach can be traced back at least to Pour El and Richards~\cite{PourElRich}.  See \cite{Brattka.Hertling.ea:08} for an excellent and reader-friendly introduction to the theory of computable Banach spaces.

  Although computable presentations of Banach spaces have been studied for several decades, the index set approach has not yet been applied to measure the complexity of the classification problem for standard subclasses of separable Banach spaces. 
Herein, we initiate the systematic study of index sets of classes of separable Banach spaces by focusing on the class of separable Lebesgue spaces (i.e. spaces that are $L^p$ spaces for some $p \geq 1$) and several of its natural subclasses.  Obviously, there are many natural questions beyond this class.
  
  \subsection{The results} 
  
 How hard is it to determine if a number indexes a presentation of a Lebesgue space?
More formally,  what is the complexity of the index set 
\[
I_{Leb} = \{e\ :\ \mbox{$e$ indexes a Lebesgue space presentation} \}? 
\]
 The complexity of this index set reflects how hard it is to \emph{characterize} or \emph{distinguish} Lebesgue spaces among all Banach spaces.
 For example, Hilbert spaces are characterized by the parallelogram law, which makes their index set merely $\Pi^0_2$ (see Lemma~\ref{lm:hilbert}).
  Is there any similar ``local'' law --  e.g., a first-order sentence --  which would capture the property of being a Lebesgue space?
 
 At first glance, the characterization problem $I_{Leb}$ seems to be no better than $\Sigma^1_1$, because we seemingly have to search for an isomorphism $f$ which may not be computable.  Indeed, this upper bound is also suggested by the characterization of $L^p$ spaces via Banach lattice relations due to Kakutani \cite{Kakutani.1941}. 
 If this crude upper bound was optimal this would imply that there is no reasonable ``local'' law which isolates Lebesgue spaces among all Banach spaces.
  Rather surprisingly, our first main result shows that the index set of Lebesgue spaces has a much lower complexity.

\begin{maintheorem}\label{mthm:Lebesgue}
  The set of all indices of Lebesgue space presentations is  $\Pi^0_3$. 
\end{maintheorem}
   
Main Theorem \ref{mthm:Lebesgue} implies that there must be a local property that distinguishes Lebesgue spaces.  What is this property?
Our proof implies that Lebesgue spaces are characterized among all Banach spaces by the success of an algorithm which attempts to build a \emph{formal disintegration} of the given space. The notion of a formal disintegration is a development of the earlier notion of disintegration~\cite{McNellp}, \cite{Clanin.McNicholl.Stull.2019}.  The associated independence property vaguely resembles $S$-independence~\cite{DoMel, DoMel1} in discrete completely decomposable groups, as well some other notions in the literature on discrete countable p-groups (c.f.~\cite{RogersAbel}).
  We believe that our proof of Main Theorem \ref{mthm:Lebesgue} has no analogy in computable analysis, while the only technical similarity with the above-mentioned results is the use of some independence notion.

  Although we do not know if the upper bound $\Pi^0_3$ is tight when the exponent $p \geq 1$ is not known, when $p$ is held fixed
we achieve a tight upper bound.

 \begin{maintheorem}\label{mthm:Lp}
  Suppose $p \geq 1$ is a computable real.  Then, the set of all indices of $L^p$ space presentations 
is $\Pi_2^0$-complete.
\end{maintheorem}

The simpler proof of Main Theorem \ref{mthm:Lp} will be given before the proof of Main Theorem \ref{mthm:Lebesgue} (they reappear as
 Theorems~\ref{thm:Lp} and~\ref{thm:pres.Lebesgue}, respectively).   The situation resembles the main results in \cite{DoMel1} where the completeness of the $\Sigma^0_7$-upper bound for completely decomposable groups is not known, but the closely related $\Delta^0_5$-categoricity bound used to establish this $\Sigma^0_7$ estimate is provably optimal.
We leave open: \emph{Is the upper bound from Main Theorem \ref{mthm:Lebesgue} tight?} 
 The difficulty that we faced in our attempts resolve this question is the lack of a procedure for computing the exponent of a Lebesgue space from an index of one of its presentations.  This issue will be 
 discussed in a forthcoming paper by the second author.  
  We suspect that new insights into continuous definability in Lebesgue spaces are required to settle these two closely related questions.  
  The proof of Main Theorem \ref{mthm:Lebesgue} sidesteps this difficulty by means of a formula due to O. Hanner \cite{Hanner.1956} for the modulus of uniform convexity of an $L^p$ space.  
 
 \medskip

We conclude our discussion with an optimal analysis of index set complexity for each individual isometric isomorphism type of $L^p$-spaces when a computable $p \geq 1$ other than $2$ is held fixed. For example, Theorem~\ref{thm:known.exp}.\ref{thm:known.exp.lpnLp01} says that for each $n \geq 1$, the set of all indices of presentations of $\ell^p_n \oplus_p L^p[0,1]$ is 
	$d$-$\Sigma_2^0$-complete.  Theorem~\ref{thm:known.exp} consists of six parts and is a bit too lengthy to be stated here; we therefore postpone its complete formulation until Section 5.
	The proof of Theorem~\ref{thm:known.exp} implements an effective functor which transforms a linear order into a measure space preserving some properties of interest; see, e.g., Corollary   \ref{cor:faithful} after Definition~\ref{def:faithful}. In several cases the functor will allow us to work with a linear order and then transform it into a Lebesgue space, therefore significantly reducing the combinatorial complexity of some parts of the proof.  This idea may lead to new applications beyond the study of index sets.

Finally, we investigate the isometric isomorphism problem for $L^p$ spaces.  Specifically, we prove the following.

 \begin{maintheorem}\label{mthm:isom} Suppose $p \geq 1$ is a computable real other than $2$.  Then, the set of all pairs $(e,e')$ so that $e,e'$ index presentations of isometrically isomorphic $L^p$ spaces is co-$3$-$\Sigma_3^0$-complete.
 \end{maintheorem}
 
Again, at first glance, the bound given by Main Theorem \ref{mthm:isom} is quite a bit lower than what might be expected as it implies that there is no need to search for an isometric isomorphism and that one may rely instead on a local property.  In fact, this property is the classical characterization of separable $L^p$ spaces (see Theorem \ref{thm:classification} below and also \cite{Cembranos.Mendoza.1997}).  To the best of our knowledge, this is the first natural example of an index set at this particular level of the relativized Ershov hierarchy.  The proof of Main Theorem \ref{mthm:isom} combines the techniques developed in the proof of Main Theorem \ref{mthm:Lp} and the linear order functor alluded to previously.

We now proceed to summarize relevant background from functional and computable analysis.

\section{Background}\label{sec:back}

\subsection{Background from functional analysis}\label{sec:back::subsec:FA}

Let $\F$ denote the field of scalars.  This can be either $\R$ or $\C$.  

When $V_0$ and $V_1$ are vectors spaces, let $V_0 \oplus V_1$ denote their external 
direct product.  Suppose $\B_0$ and $\B_1$ are Banach spaces.  
Then, $\B_0 \oplus_p \B_1$ consists of the vector space $\B_0 \oplus \B_1$ together with the 
norm defined by 
\[
\norm{(v_0,v_1)}^p = \norm{v_0}_{\B_0}^p + \norm{v_1}_{\B_1}^p.
\]
$\B_0 \oplus_p \B_1$ is called the \emph{$L^p$-sum} of $\B_0$ and $\B_1$.  
When $\{\B_j\}_{j \in \N}$ is a sequence of Banach spaces, the $L^p$-sum of $\{B_j\}_{j \in \N}$ is defined to be the set of all 
functions $f$ in the infinite Cartesian product $\prod_j \B_j$ so that 
$\sum_j \norm{f}_{\B_j}^p < \infty$; we denote this sum by $\oplus_p \{B_j\}_{j \in \N}$.  
The $L^p$ sum of a sequence of Banach spaces is again a Banach space under component-wise vector addition and scalar multiplication and with the norm defined by 
\[
\norm{f}^p = \sum_j \norm{f(j)}_{\B_j}^p.
\]

It is well-known that every nonzero separable $L^2$ space is isometrically isomorphic either to $\ell^2$ or $\ell^2_n$ for some $n$.  For $L^p$ spaces with $p \neq 2$, we rely extensively on the following classification, a proof of which can be found in \cite{Cembranos.Mendoza.1997}.  

\begin{theorem}[Classification of separable $L^p$ spaces]\label{thm:classification}
Suppose $1 \leq p < \infty$ and $p \neq 2$.  
Then, every nonzero separable $L^p$ space is isometrically isomorphic to exactly one of the following.
\begin{enumerate}
	\item $\ell^p_n$ for some $n \geq 1$.  In this case, the underlying measure space is purely atomic and has exactly $n$ atoms.
	
	\item $\ell^p$.  In this case, the underlying measure space is purely atomic and has $\aleph_0$ atoms.
	
	\item $L^p[0,1]$.  In this case, the underlying measure space is non-atomic.
	
	\item $\ell^p_n \oplus_p L^p[0,1]$ for some $n \geq 1$.  In this case, the underlying measure space
	has exactly $n$ atoms but is not purely atomic.
	
	\item $\ell^p \oplus_p L^p[0,1]$.  In this case, the underlying measure space
	has $\aleph_0$ atoms but is not purely atomic.
\end{enumerate}
\end{theorem}

\subsection{Background from computable analysis}\label{sec:back::subsec:CA}

Let $\mathcal{B}$ be a Banach space.  A set of vectors $S\subseteq \B$ is said to be \textit{linearly dense in $\mathcal{B}$} if $\mathcal{B}$ is the closure of the linear span of the vectors in $S$.

\begin{definition}
A function $R : \N \rightarrow \mathcal{B}$ is a \emph{structure} on $\mathcal{B}$ if its range is linearly dense in $\mathcal{B}$. 
 If $R$ is a structure on $\mathcal{B}$, then $(\mathcal{B}, R)$ is a \emph{presentation} of $\mathcal{B}$ and $R(n)$ is called the \emph{$n$-th distinguished point} of $(\mathcal{B}, R)$.  
\end{definition}
Thus, to define a presentation of a Banach space, it suffices to specify the distinguished points.  
If $\B^\#$ is a Banach space presentation, then each object associated with $\B$ is also naturally associated with $\B^\#$.  We will therefore sometimes identify a structure on a Banach space with the associated presentation.

A Banach space may have a presentation that is designated as \emph{standard}; such a space is identified with its standard presentation.  The standard presentations of the $L^p$ spaces are defined as follows.  To begin, the standard presentation of $\ell^p$ 
is defined by taking the $n$-th distinguished point to be $e_n$ (the $n$-th standard basis vector).
The standard presentation of $L^p[0,1]$ is defined by taking the $n$-th distinguished point to be the indicator function of the $n$-th dyadic interval (in a standard enumeration).
The standard presentations of $\ell^p_n$, $\ell_n^p \oplus_p L^p[0,1]$, and $\ell^p \oplus_p L^p[0,1]$ are formed similarly.

A presentation of a Banach space induces associated classes of rational vectors and rational balls as follows.  Recall that $\F =\R$ or $\C$. Let $\F_\Q = \F \cap \Q(i)$; we refer to the elements of $\F_\Q$ as \emph{rational scalars}.
Suppose $\mathcal{B}^\#$ is a presentation of $\mathcal{B}$.  We say $v \in \mathcal{B}$ is a \emph{rational vector} of $\mathcal{B}^\#$ if it is a rational linear combination of distinguished points of $\B^\#$; that is if $v$ can be written as $\sum_{j = 1}^n \alpha_j u_j$ where each $\alpha_j$ is a rational scalar and each $u_j$ is a distinguished vector.  A \emph{rational open ball of $\mathcal{B}^\#$} is an open ball whose center is a rational vector of $\mathcal{B}^\#$ and whose radius is a positive rational number.  

We can define a coding of the rational vectors of $\B^\#$ by means of a G\"odel numbering of the formal expressions that represent rational vectors.  Accordingly, let $\Rat(\B^\#;n)$ denote the $n$-th rational vector of $\B^\#$ under this coding.  This coding has the feature that for all $\alpha \in \F_\Q$ and $m,n \in \N$, a code of $\alpha \Rat(\B^\#;m) + \Rat(\B^\#;n)$ can be computed from $m,n,\alpha$ independently of $\B^\#$.  
Similarly, we can code rational open balls in such a way that from a code of a rational ball we can compute its radius and a code of its center independently of the presentation.  Let $B(\B^\#;n)$ denote the $n$-th rational ball of $\B^\#$ under this coding, and let $I_n = B(\R, n)$.  

The definition below is standard; see e.g., \cite{PourElRich}.

\begin{definition}
We say that a presentation $\mathcal{B}^\#$ of a Banach space is \emph{computable} if the norm functional is computable on the set of rational vectors of $\mathcal{B}^\#$. 
\end{definition}
 That is, $\mathcal{B}^\#$ is computable if there is an algorithm that given a (code of a) rational vector $v$ of $\mathcal{B}^\#$ and a nonnegative integer $k$, computes a rational number $q$ so that $|q - \norm{v}| < 2^{-k}$.  An index of such an algorithm is referred to as an 
\emph{index} of $\mathcal{B}^\#$.  
Clearly,  if $\mathcal{B}_0$ and $\mathcal{B}_1$ are isometrically isomorphic, then every index of a presentation of $\mathcal{B}_0$ is also an index of a presentation of $\mathcal{B}_1$.  

Recall that a Polish metric space is computably presentable if it possesses a dense countable sequence upon which the distance function is uniformly computable~\cite{MelIso}.
A computable presentation of a Banach space can be also viewed as a computable presentation of a Polish space (under the metric induced by the norm) with respect to which the standard Banach space operations become uniformly computable operators; see \cite{MelIso}. The latter can be taken for an equivalent definition of a computable structure on a Banach space. 
  Note that in this setting, computability of the norm does not necessarily imply computability of the operations; see \cite{MelNg} for a detailed analysis of this phenomenon.

\begin{definition}\label{def:diagram.name}
If $\B^\#$ is a presentation of a Banach space, then the \emph{diagram} of $\B^\#$ consists 
of all pairs $(m,n)$ so that $\norm{\mathcal{R}(\B^\#; m)}_\B \in I_n$.   
\end{definition}

Note that a presentation is computable if and only if its diagram  is a computably enumerable set. 
By means of a standard coding, we may identify the diagram of a Banach space presentation with a set of natural numbers.  

In order to state our results about index sets in a highly uniform manner, we introduce names of Banach space presentations as follows.  A \emph{name} of a Banach space presentation $\B^\#$ is a function $f \in \N^\N$ that enumerates the diagram of $\B^\#$.  Note that if $\B_0$ and $\B_1$ are isometrically isomorphic, then any name of a presentation of $\B_0$ is also a name of a presentation of $\B_1$.  

Names and indices of presentations are related as follows.  There is a computable $f : \N \rightarrow \N$ so that for all $e \in \N$, $\phi_{f(e)}$ names a Banach space presentation (possibly the zero space) and if $e$ indexes a Banach space presentation $\B^\#$, then this presentation is named by 
$\phi_{f(e)}$.  Thus any result we prove about complexity of name sets immediately yields the same result mutatis mutandis about the complexity of index sets.

We conclude this section by pinning down the complexity of the set of all names of Banach space presentations.  We first prove the following lemma which will be useful later as well.  

\begin{lemma}\label{lm:Banach.complete}
Suppose $\B^\#$ is a computable Banach space presentation, and assume $P \subseteq \N^\N$ is $\Pi_2^0$.  Then, there is a computable operator $F : \N^\N \rightarrow \N^\N$ so that 
for every $f \in \N^\N$, $F(f)$ names $\B^\#$ if $P(f)$, and otherwise $F(f)$ names no Banach space presentation.
\end{lemma}

\begin{proof}
Let $D$ denote the diagram of $\B^\#$.  
Note that $D$ is infinite.  
Fix a computable predicate $Q \subseteq \N^\N \times \N$ so that for all $f \in \N^\N$, 
\[
P(f) \Leftrightarrow \exists^\infty x Q(f; x).
\]
It is straightforward to construct a computable $F : \N^\N \rightarrow \N^\N$ so that for all $f \in \N^\N$, $F(f)$ enumerates $D$ if $P(f)$ and otherwise enumerates a finite subset of $D$. 
\end{proof}

\begin{theorem}\label{thm:Banach.complete}
The set of all names of Banach space presentations is $\Pi_2^0$ complete.
\end{theorem}

\begin{proof}[Proof sketch:] 
The lower bound is established by Lemma \ref{lm:Banach.complete}.
A code of a rational vector may be regarded as a code of a finite sequence of rational scalars which in turn may be regarded as a vector in $C_{00}(\F_\Q)$ (the set of all finitely-supported infinite sequences of rational scalars).  The key idea to obtaining the upper bound is to observe that $f \in \N^\N$ names a Banach space presentation if and only if it induces a seminorm on $C_{00}(\F_\Q)$.  To be more precise, for each $f \in \N^\N$ and $g \in C_{00}(\F_\Q)$, let:
\begin{eqnarray*}
\eta_f^+(g) & = & \inf\{r \in \Q\ :\ \exists n \in \N\ \langle \langle g \rangle, n\rangle \in \ran(f)\ \wedge\ \max(I_n) < r\}\\
\eta_f^-(g) & = & \sup\{r \in \Q\ :\ \exists n \in \N \ \langle  \langle g \rangle, n\rangle \in \ran(f)\ \wedge\ r < \min(I_n)\}\\
\end{eqnarray*}
(Here $\langle, \rangle$ denotes a standard effective coding of $C_{00}(\F_\Q)$.) 
It is fairly straightforward to verify that $f$ names a Banach space presentation if and only if
$\eta^+_f = \eta_f^-$ and $\eta_f := \eta_f^+$ is a seminorm on $C_{00}(\F_\Q)$ (in which case $f$ names the completion of the quotient space).
\end{proof}

\section{Preliminaries}\label{sec:prelim} 

\subsection{Formal disintegration}\label{subsec:dis} We introduce a new notion of independence which will be crucial throughout the rest of the paper.

\begin{definition}
Suppose $1 \leq p < \infty$, $\mathcal{B}$ is a Banach space, and $v_1, \ldots, v_n \in \mathcal{B}$. We say $v_1, \ldots, v_n$ are \emph{$L^p$-formally disjointly supported} if 
	\[
	\norm{\sum_j \alpha_j v_j}_\B^p = \sum_j |\alpha_j|^p \norm{v_j}_\B^p
	\]
	for all scalars $\alpha_1, \ldots, \alpha_n$.
If $f_1, \ldots, f_n \in L^p(\Omega)$ are disjointly supported, then they are $L^p$-formally disjointly supported. 
We say that $u$ is an \emph{$L^p$-formal component} of $v$ if 
$v - u$ and $u$ are $L^p$-formally disjointly supported; in this case we write 
$u \preceq v$.  
\end{definition}

In 1958, J. Lamperti proved the following remarkable result.

\begin{theorem}\label{thm:lamperti}
Suppose $1 \leq p < \infty$ and $p \neq 2$.  If $f,g$ are vectors in an $L^p$ space, then $f$ and $g$ are disjointly supported if and only if $\norm{f + g}_p^p + \norm{f - g}_p^p = 2(\norm{f}_p^p + \norm{g}_p^p)$.  
\end{theorem}
Thus, if $p \neq 2$, then $L^p$-formally disjointly supported vectors in $L^p(\Omega)$ are disjointly supported.  Another consequence of Lamperti's results is the following. 

\begin{theorem}\label{thm:disj.supp}
If $1 \leq p < \infty$, and if $p \neq 2$, then every linear isometric map from an $L^p$ space to an $L^p$ space preserves disjointness of support.
\end{theorem}

We now generalize some of the definitions from \cite{McNellp} and \cite{Clanin.McNicholl.Stull.2019}.  Suppose $\mathcal{B}$ is a Banach space.  A \emph{vector tree} of $\B$ is an injective map 
$\phi : \subseteq \omega^{<\omega} \rightarrow \B$ so that $\dom(\phi)$ is a tree. 
Let $\phi$ be a vector tree of $\B$, and let $S = \dom(\phi)$.   Each $v \in \ran(\phi)$ is referred to as a \emph{vector of $\phi$}.   If $v \not \in \ran(\phi)$, then we say $\phi$ \emph{omits $v$}.  We say that:

\begin{itemize}
	\item  $\phi$ is \emph{summative} if for every $\nu \in S$, $\phi(\nu) = \sum_{\nu'} \phi(\nu')$ where $\nu'$ ranges over all the children of $\phi$ in $S$;

	\item  $\phi$ is \emph{$L^p$-formally separating} if 
$\phi(\nu_0), \ldots, \phi(\nu_n)$ are $L^p$-formally disjointly supported whenever $\nu_0, \ldots, \nu_n \in S$ are incomparable.

	\item $\phi$ is a \emph{$L^p$-formal disintegration} of $\mathcal{B}$ if it is
summative, $L^p$-formally separating, omits $\zerovec$, and the range of $\phi$ is linearly dense.
\end{itemize}

We note that the empty map is the only disintegration of a Banach space whose only vector is its zero vector.  Disintegrations are the backbone of the analysis of computable presentations of $L^p$ spaces in \cite{McNellp}, \cite{Clanin.McNicholl.Stull.2019}, and \cite{Brown.McNicholl.2018} as well as the degrees of isometry of $\ell^p$ spaces \cite{McNicholl.Stull.2019}.  Informally, a disintegration of an $L^p$-space allows one to view the space as a tree.   Compare this to tree-bases of abelian groups \cite{RogersAbel} and tree-representations of Boolean algebras \cite{Gon}. The crucial difference of our notion above with these two notions is that we deal with uncountable normed spaces, while the former two notions are useful only for countable discrete groups and Boolean algebras, respectively.

By the above-mentioned result of Lamperti, if $p \neq 2$, then in $L^p(\Omega)$ the notion above becomes equivalent to the notion of disintegration introduced in \cite{McNellp} and \cite{Clanin.McNicholl.Stull.2019}.
 Accordingly, if $\mathcal{B}$ is an $L^p$ space, then we omit the adjective `$L^p$-formal' from these terms. 
 The main advantage of our new notion is that it does not refer to the measure space at all and therefore it makes sense for an arbitrary Banach space.

\subsection{Properties of formal disintegrations}

Note that a disintegration $\psi$ of an $L^p$ space is \emph{antitone} in the following sense: if 
$\nu_0, \nu_1 \in \dom(\psi)$, and if $\nu_0 \subseteq \nu_1$, then $\psi(\nu_1) \preceq \psi(\nu_0)$.  

The following definitions are from \cite{Clanin.McNicholl.Stull.2019} and \cite{McNellp} respectively.

\begin{definition}\label{def:isom}
Suppose $\phi_j$ is a vector tree of $\mathcal{B}_j$ for each $j\in\{0,1\}$.  Let $S_j = \dom(\phi_j)$.  A map $f : S_0 \rightarrow S_1$ is an \emph{isomorphism} of $\phi_0$ and $\phi_1$ if 
it is an order isomorphism (with respect to $\subseteq$) of $S_0$ onto $S_1$ and if $\norm{\phi_1(f(\nu))}_{\B_1} = \norm{\phi_0(\nu)}_{\B_0}$ for all $\nu \in S_0$.  
\end{definition}

\begin{definition}\label{def:anm.chain}
Suppose $\phi$ is a disintegration of an $L^p$ space, and let $S = \dom(\phi)$.  A chain $C \subseteq S$ is \emph{almost norm-maximizing} if whenever $\nu \in C$ is a nonterminal node of $S$, $C$ contains a child $\nu'$ of $\nu$ so that 
\[
\max_\mu \norm{\phi(\mu)}_p^p \leq \norm{\phi(\nu')}_p^p + 2^{-|\nu|}  
\]
where $\mu$ ranges over the children of $\nu$ in $S$.  
\end{definition}

The next proposition is from \cite{Clanin.McNicholl.Stull.2019}

\begin{proposition}\label{prop:subvectorLimitsExist} If $g_0 \succeq g_1 \succeq ...$ are vectors in an $L^p$ space, then $\lim_n g_n$ exists in the $L^p$-norm and is the $\preceq$-infimum of $\{g_0,g_1,...\}$.
\end{proposition}

The following theorem was first proven for $\ell^p$ spaces in \cite{McNellp} and generalized to arbitrary $L^p$ spaces in \cite{Brown.McNicholl.2018}.

\begin{theorem}\label{thm:limits.chains}
Suppose $\Omega$ is a measure space and $\phi: S \rightarrow L^p(\Omega)$ is a disintegration.\\
\begin{enumerate}
	\item If $C \subseteq S$ is an almost norm-maximizing chain, then the $\preceq$-infimum of $\phi[C]$ exists and is either \textbf{0} or an atom of $\preceq$.  Furthermore, the $\preceq$-infimum of $\phi[C]$ is the limit in the $L^p$ norm of $\phi(\nu)$ as $\nu$ traverses the nodes in $C$ in increasing order.\label{thm:limits.chains::itm:inf}
		
	\item If $\{C_n\}_{n < \kappa}$ is a partition of $S$ into almost norm-maximizing chains (where $\kappa \leq \omega$), then the $\preceq$-infima of $\phi[C_0], \phi[C_1], ...$ are disjointly supported.  Furthermore, if $A$ is an atom of $\Omega$, then there exists a unique $n$ so that $A$ is the support of the $\preceq$-infimum of $\phi[C_n]$. \label{thm:limits.chains::itm:unique} 
\end{enumerate}
\end{theorem}

The next theorem generalizes a result on isomorphisms of disintegrations from \cite{Clanin.McNicholl.Stull.2019}.  Although disintegrations are not bases, this theorem states an important way in which they behave like bases.
 
\begin{theorem}\label{thm:lifting}
Suppose $\phi_j$ is a $L^p$-formal  disintegration of $\mathcal{B}_j$ for each $j \in \{0,1\}$, and suppose 
$f$ is an isomorphism of $\phi_0$ with $\phi_1$.  Then, there is a unique isometric isomorphism 
$T_f$ of $\mathcal{B}_0$ onto $\mathcal{B}_1$ so that $T_f(\phi_0(\nu)) = \phi_1(f(\nu))$ for all 
$\nu \in \dom(\phi_0)$.  
\end{theorem}

\begin{proof}
Let $S_0 = \dom(\phi_0)$, and let $L$ denote the linear span of the vectors of $\phi_0$.  
For each $n \in \N$, let $F_n = S_0 \cap \omega^{\leq n}$, and let $L_n$ denote the linear span 
of $\phi_0[F_n]$.  For each $v \in L$, there is a least $n \in \N$ so that $v \in L_n$; denote this number by $n_v$. 

When $\nu$ is a leaf node of $F_n$, call $\phi_0(\nu)$ a \emph{leaf vector} of $F_n$.  Since 
$\phi_0$ is non-vanishing and $L^p$-formally separating, the 
leaf vectors of $F_n$ are linearly independent.  Because $\phi_0$ is summative, each $v \in L_n$ can be expressed as a linear combination of the leaf vectors of $F_n$ in exactly one way.

When $v \in L$, let $T(v) = \sum_\nu \beta_\nu \phi_1(f(\nu))$ where $\nu$ ranges over the leaf nodes of $F_{n_\nu}$ and $v = \sum_\nu \beta_\nu \phi_0(\nu)$.  Then, $T$ is well-defined and linear.  Since $f$ is an isomorphism, and because $\phi_0$ and $\phi_1$ are $L^p$-formally separating, $T$ is isometric.  Hence, since $\ran(\phi_0)$ is dense in $\B_0$, $T$ has a unique isometric extension to $\mathcal{B}_0$, and this extension is linear.  We denote this extension by $T$ as well.  Since $f$ is surjective and the vectors of $\phi_1$ are linearly dense in $\mathcal{B}_1$, $T$ is surjective.  The uniqueness of $T$ follows from the linear density of the vectors of $\phi_0$.   
\end{proof}

The following is crucial to our analysis of index sets of Lebesgue space presentations.

\begin{theorem}\label{thm:disint.Lp}
A Banach space has an $L^p$-formal disintegration if and only if it is isometrically isomorphic to a separable $L^p$ space.  
\end{theorem}

\begin{proof}
The converse follows from the uniformity of the proof of Theorem \ref{thm:disint.comp} below.  

  So, let $\mathcal{B}$ be a Banach space, and suppose $\phi$ is an $L^p$-formal disintegration of $\mathcal{B}$.  Let $S = \dom(\phi)$.  We may assume $\B$ is nonzero.  
Without loss of generality, assume $\norm{\phi(\emptyset)}_\mathcal{B} = 1$.  

We first associate each node of $S$ with a subinterval of $[0,1]$ as follows.  
Let $I_\emptyset = [0,1]$.  Let $\nu$ be a non-root node of $S$, and suppose $I_{\nu'}$ has been defined for all $\nu' \in S$ that lexicographically precede $\nu$.  
If $\nu$ is the lexicographically least child of $\nu^-$, then we define the left endpoint of $I_\nu$ to be the left endpoint of $I_{\nu^-}$.   Suppose $\nu$ is not the lexicographically least child of $\nu^-$.  Let $\nu'$ denote the lexicographically largest sibling of $\nu$ that lexicographically precedes $\nu$.  We then define 
the left endpoint of $I_\nu$ to be the right endpoint of $I_{\nu'}$.  
In either case, we define the right endpoint of $I_\nu$ to be 
$a + \norm{\phi(\nu)}_\mathcal{B}^p$ where $a$ is the left endpoint of $I_\nu$.

Let $\mathcal{M}$ denote the $\sigma$-algebra generated by the $I_\nu$'s, and define $\mu$ to be the restriction of Lebesgue measure to $\mathcal{M}$.  
Let $\Omega = ([0,1], \mathcal{M}, \mu)$.  
Let $\psi(\nu)$ denote the indicator function of $I_\nu$ for each $\nu \in S$.  Then, $\psi$ is a disintegration of $L^p(\Omega)$, and the identity map is an isomorphism of $\phi$ with $\psi$.  Thus, by Theorem \ref{thm:lifting}, $\mathcal{B}$ is isometrically isomorphic to $L^p(\Omega)$.
\end{proof}

The following will be used in our analysis of the index sets of presentations of $\ell^p_n$.

\begin{lemma}\label{lm:antichains}
Suppose $\psi$ is a disintegration of an $L^p$ space $\B$.  
\begin{enumerate}
	\item $\B$ is finite-dimensional if and only if there is a bound $n$ such that every antichain of $\dom(\psi)$ has size no greater than $n.$
	
	\item If $\B$ is finite-dimensional, then the dimension of $\B$ is the least $n \in \N$ so that 
	$\dom(\psi)$ 
	does not contain an antichain of size $n+1$.
\end{enumerate}
\end{lemma}

\begin{proof}
 Without loss of generality, suppose $\B$ is nonzero.  
Let $S = \dom(\psi)$.  We then observe that every disjointly supported set of nonzero vectors is linearly independent.  Thus, since $\psi$ is separating, if $\B$ is finite-dimensional, then there is a bound $n$  (namely the dimension of $\B$) such that every antichain of $S$ has size no greater than $n$.  
Conversely, suppose such a bound exists.  Then, there is a largest $n \in \N$ so that $S$ contains an antichain of size $n$; let $F $ be such an antichain.  

We claim that each node of $F$ is a terminal node of $S$.  By way of contradiction, suppose $\nu \in F$ is nonterminal, and let $\nu_0 \in S$ be a child of $\nu$.  Since $\psi$ is injective and summative, $S$ must contain another child of $\nu$, $\nu_1$.  
Thus, $(F - \{\nu\}) \cup \{\nu_0, \nu_1\} $ is an antichain of $S$ of size $n + 1$ which is a contradiction.  Thus, every node of $F$ is a terminal node of $S$. 

By the maximality of $n$, every node of $S$ is comparable to at least one node of $F$.   Since $\psi$ is summative, we conclude that $\B$ is the closed linear span of $\psi[F]$.  Since $F$ is finite, 
$\B$ is the linear span of $\psi[F]$.  Hence, the dimension of $\B$ is $n$.
\end{proof}

\subsection{The modulus of uniform convexity} Suppose we are given a computable
presentation of a Lebesgue space, and our task is to extract its exponent $p$.  According to our definitions, a computable presentation does not contain any information about $p$.
Therefore, we aim to find a way of using only the norm and the Banach space operations to approximate $p$ with an arbitrary precision.
The following local parameter, which has a long history in the geometric analysis of Banach spaces, will be very helpful.

\begin{definition}\label{def:mod.conv}
Suppose $\mathcal{B}$ is a nonzero Banach space.  For all $0 \leq \epsilon \leq 2$, let
\[
\delta_\mathcal{B}(\epsilon) = \inf_{u,v} 1 - \frac{\norm{u + v}}{2}
\]
where $u,v$ range over all unit vectors of $\mathcal{B}$ so that $\norm{u - v} \geq \epsilon$. 
The function $\delta_\B$ is called the \emph{modulus of uniform convexity} of $\B$.
\end{definition}

It is well-known that if $\B$ is an $L^p$ space with $1 <p < \infty$, then $\delta_\B$ is positive.  
Later, in the proof of Main Theorem \ref{mthm:Lebesgue}, we will use the modulus of uniform convexity to 
produce approximations of the exponent of a Lebesgue space from one of its presentations.  
This is made possible by an explicit formula for the modulus of convexity of $L^p$ spaces due to O. Hanner \cite{Hanner.1956}.  In order to state this formula, we first make the following definition.

\begin{definition}\label{def:delta.func}
Suppose $1 < p < \infty$ and $0 < \epsilon \leq 2$.  
\begin{enumerate}
	\item If $1 < p \leq 2$, then let $\delta(p,\epsilon)$ denote the unique number $\delta \in [0,1]$ so that 
	\[
	\left( 1 - \delta + \frac{\epsilon}{2}\right)^p + \left| 1 - \delta - \frac{\epsilon}{2}  \right|^p = 2.
	\]
	
	\item If $2 \leq p$, then let 
	\[
	\delta(p, \epsilon) = 1 - \left( 1 - \left( \frac{\epsilon}{2} \right)^p \right)^{1/p}.
	\]
\end{enumerate}
\end{definition}

We can now state Hanner's Theorem (which in fact he attributes to A. Beurling).  

\begin{theorem}[Hanner 1956 \cite{Hanner.1956}]\label{thm:Hanner}
Suppose $1 < p < \infty$ and $\mathcal{B}$ is either $\ell^p$ or $L^p[0,1]$.  Then, 
$\delta_\mathcal{B}(\epsilon) = \delta(p, \epsilon)$ whenever $0 < \epsilon \leq 2$.  
\end{theorem}

For the sake of computation and approximation, it will be useful to show that Hanner's formula
applies to a broader class of $L^p$ spaces and that, at least for these spaces, the weak inequality in
Definition \ref{def:mod.conv} can be made strict when $\epsilon < 2$.  

\begin{proposition}\label{prop:hanner.extended}
Suppose $1 < p < \infty$, and let $\mathcal{B}$ be a separable $L^p$ space whose dimension is at least $2$.  Then, 
\begin{enumerate} 
	\item $\delta_B(\epsilon) = \delta(p, \epsilon)$ for all $\epsilon \in (0,2]$.  

	\item If $0 < \epsilon < 2$,  
\[
\delta_\B(\epsilon) = \inf_{u,v} 1 - \frac{1}{2} \norm{u - v}_p
\]
where $u,v$ range over all unit vectors of $\B$ so that $\norm{u - v}_p > \epsilon$.
\end{enumerate}
\end{proposition}

\begin{proof}
Let:
\begin{eqnarray*}
u(\epsilon) & = &
 \left\{ \begin{array}{cc}
 2^{-1/p}(1 - \delta(p,\epsilon) + \frac{\epsilon}{2}, 1 - \delta(p, \epsilon) - \frac{\epsilon}{2}) & 1 < p < 2\\
 (1 - \delta(p,\epsilon), \frac{\epsilon}{2}) & 2 \leq p\\
 \end{array}
 \right. \\
 v(\epsilon) & = &  \left\{ \begin{array}{cc}
 2^{-1/p}(1 - \delta(p,\epsilon) - \frac{\epsilon}{2}, 1 - \delta(p, \epsilon) + \frac{\epsilon}{2}) & 1 < p < 2\\
 (1 - \delta(p,\epsilon), - \frac{\epsilon}{2}) & 2 \leq p\\
 \end{array}
 \right. 
\end{eqnarray*}
Then, $u(\epsilon)$ and $v(\epsilon)$ are unit vectors of $\ell^p_2$ so that 
$\norm{u(\epsilon) - v(\epsilon)}_p = \epsilon$ and $\delta(p, \epsilon) = 1 - \frac{1}{2} \norm{u(\epsilon) + v(\epsilon)}_p$.  

We first show $\delta_\B(\epsilon) = \delta(p, \epsilon)$.  Since the dimension of $\B$ is at least $2$, 
$\ell^p_2$ isometrically embeds into $\B$.  Thus, $\delta_\B \leq \delta_{\ell^p_2}$.  
However, $\delta_{\ell^p_2}(\epsilon) \leq 1 - \frac{1}{2}\norm{u(\epsilon) + v(\epsilon)}_p = \delta(p, \epsilon)$.   And, since $\B$ isometrically embeds into $L^p[0,1]$, $\delta_{L^p[0,1]}(\epsilon) \leq \delta_\B$.  So, by Theorem \ref{thm:Hanner}, $\delta(p,\epsilon) \leq \delta_\B(\epsilon)$ and so $\delta(p, \epsilon) = \delta_\B(\epsilon)$.

Now, suppose $0 < \epsilon < 2$, and let $\delta_\B'(\epsilon) = \inf_{u,v} 1 - \frac{1}{2}\norm{u + v}_p$
where $u,v$ range over all unit vectors of $\B$ so that $\norm{u - v}_p > \epsilon$.  Then, 
$\delta_\B(\epsilon) \leq \delta_\B'(\epsilon)$.  Fix an isometric embedding $T$ of $\ell^p_2$ into $\B$.  Then, when $\epsilon < \epsilon' < 2$, 
\[
\delta_\B'(\epsilon) \leq 1 - \frac{1}{2} \norm{T(u(\epsilon')) + T(v(\epsilon'))}_p = \delta(p, \epsilon').
\]
Since $\delta$ is continuous, $\delta_\B'(\epsilon) \leq \delta(p, \epsilon) = \delta_\B(\epsilon)$.  
\end{proof}

The proposition below will later (Lemma~\ref{lm:comp.exp}) allow us to approximate $p$ based on an approximation of  $\delta_B(\epsilon) = \delta(p, \epsilon)$.
Although this approximation will not be computable in general, it will be $\Delta^0_2$ which is sufficient for our purposes.  In particular, this approximation will be essential in Section \ref{sec:index.Lebesgue} where we show that the index set of all computable Lebesgue space presentations is $\Pi^0_3$.

\begin{proposition}\label{prop:delta.inc.dec}
Fix $0 < \epsilon < 2$.  
\begin{enumerate}
	\item If $2 \leq p_1 < p_2 < \infty$, then $\delta(p_2, \epsilon) < \delta(p_1, \epsilon)$. \label{prop:deltap.inc.dec::itm:p>2}

	\item If $1 < p_1 < p_2 \leq 2$, then $\delta(p_1, \epsilon) < \delta(p_2, \epsilon)$.\label{prop:deltap.inc.dec::itm:p<2}
\end{enumerate}
\end{proposition}

\begin{proof}
(\ref{prop:deltap.inc.dec::itm:p>2}): It is sufficient to calculate

 $\frac{d}{dp} \left( 1 - \left( \frac{\epsilon}{2} \right)^p \right)^{1/p} $
$= \dfrac{(1-(\epsilon/2)^p)^{1/p} \left[ ( 2^p - \epsilon^p) \ln(1-(\epsilon/2)^p) + p \epsilon^p \ln(\epsilon/2)\right]}{p^2(\epsilon^p -2^p)} $ 
 
and see that  the result is positive for a positive $p$ whenever  $0 < \epsilon < 2$. \\

(\ref{prop:deltap.inc.dec::itm:p<2}): Suppose $1 < p_1 < p_2 \leq 2$.  
Set $\delta = \delta(p_1, \epsilon)$.  Set $h_p(s,t) = (s + t)^p + |s - t|^p$ when $s,t \geq 0$ and $p > 1$.  
It follows that $h_p$ is increasing in each variable (divide by larger of two and differentiate the result).  
Also, $h_p(1 - \delta(p, \epsilon), \frac{\epsilon}{2}) = 2$ whenever $1 < p \leq 2$.  

Set $\delta = \delta(\epsilon, p_1)$.  Then, 
\[
\left( \frac{(1 - \delta + \frac{\epsilon}{2})^{p_1} + |1 - \delta + \frac{\epsilon}{2}|^{p_1}}{2} \right)^{p_2/p_1} = 1
\]
Since $p_2 > p_1$, $x \mapsto x^{p_2/p_1}$ is strictly convex.  We infer that, 
\[
\frac{1}{2}\left[ (1 - \delta + \frac{\epsilon}{2})^{p_2} + |1 - \delta + \frac{\epsilon}{2}|^{p_2} \right] > 1
\]
Since $h_{p_2}$ is increasing in both variables, it follows that $\delta(p_2, \epsilon) > \delta(p_1, \epsilon)$.
\end{proof}

\subsection{Computable disintegrations}\label{sec:prelim::subsec:CA}

We begin by stating two theorems from prior work on the computation of disintegrations and almost norm-maximizing chains.  Theorem \ref{thm:disint.comp} is from \cite{Clanin.McNicholl.Stull.2019}. 
Theorem \ref{thm:decomp} was first proven for $\ell^p$ spaces in \cite{McNellp} and for general $L^p$ spaces in \cite{Brown.McNicholl.2018}

\begin{theorem}\label{thm:disint.comp}
Suppose $p$ is a computable real so that $p \geq 1$ and $p \neq 2$.  
If $\mathcal{B}^\#$ is a computable presentation of an $L^p$ space, then 
there is a computable disintegration of $\mathcal{B}^\#$.
\end{theorem}

\begin{theorem}\label{thm:decomp}
If $\B^\#$ is a computable presentation of an $L^p$ space, and if $\phi$ is a computable disintegration of $\B^\#$, then there is a partition $\{C_n\}_{n < \kappa}$ of $\dom(\phi)$ into uniformly c.e. almost norm-maximizing chains.
\end{theorem}

The key feature of the proofs of these theorems is their high degree of uniformity.  
To be more precise about this, we introduce names of disintegrations and almost norm-maximizing chain decompositions as follows.  
Let $\psi$ be a vector tree of $\B^\#$.  A \emph{name} of $\psi$ is an enumeration of the set of all 
finite subsets of $\{(\nu, n)\ :\ \psi(\nu) \in B(\B^\#;n)\}$.  Suppose $\psi$ is a disintegration of $\B^\#$, and let $\mathcal{C} = \{C_n\}_{n < \kappa}$ is a decomposition of $\dom(\psi)$ into almost norm-maximizing chains.  
A \emph{name} of $\mathcal{C}$ is an enumeration of the set of all finite subsets of 
$\{(n, \nu)\ :\ \nu \in C_n\}$.  After suitable coding, these names can (and will) be regarded as functions in $\N^\N$.  

As noted above (as well as in \cite{Clanin.McNicholl.Stull.2019}), the proof of Theorem \ref{thm:disint.comp} is uniform.  That is, it is possible to uniformly compute a name of a disintegration of $\mathcal{B}^\#$ from a name of $\mathcal{B}^\#$ and a name of $p$. 
More formally,  there is a computable operator $\Phidis : (\N^\N)^2 \rightarrow \N^\N$ so that 
for all $f,g \in \N^\N$, if $g$ names a real $p \geq 1$ so that $p \neq 2$, and if 
$f$ names a presentation $\B^\#$ of an $L^p$ space, then $\Phidis(f,g)$ names 
a disintegration of $\B^\#$.
The proof of Theorem \ref{thm:decomp} is similarly uniform, and so there is a computable operator $\Phidecomp:(\N^\N)^3\rightarrow \N^\N$ so that for all $f,g,h\in \N^\N$ if $f$ names an $L^p$-space presentation $\B^\#$, if $g$ names a real $p \geq 1$ so that $p \neq 2$, and if $h$ names a disintegration $\phi$ of $\B^\#$, then $\Phidecomp(f,h,g)$ names an almost norm-maximizing chain decomposition of $\dom(\phi)$.
We will be using these observations and operators throughout the rest of the paper.

\subsection{The language of finite approximations} 

A function $f \in \N^\N$ may or may not be a name of a disintegration (or of a decomposition, etc.).  
Several of the forthcoming proofs require us to reason about the objects that such a function may name.  
In order to facilitate this reasoning, we introduce a formal language of presentations, disintegrations, and decompositions as follows. 

Let $\Lpres$ denote the language consisting of the following.
\begin{enumerate}
	\item Distinct constants $\overline{v_0}, \overline{v_1}, \ldots$.  
	\item A binary operation symbol $+$.
	\item For each rational scalar $\alpha$, a unary function symbol $\cdot_\alpha$.
	\item For each positive rational number $r$, unary predicates $P_{<,r}$ and $P_{>r}$.
\end{enumerate}
We write $\cdot_\alpha x$ as $\alpha x$, $P_{<,r}(x)$ as $\norm{x} < r$ and 
$P_{>,r}(x)$ as $\norm{x} > r$. 

Let $f \in \N^{\leq \N}$.  (Here $f$ should be thought of a possible name of a Banach space presentation.)  Write $f \models \norm{\sum_{j \leq M} \alpha_j \overline{v_j}} < r$ if there exists $n$ so that 
$(\langle \alpha_0, \ldots, \alpha_M \rangle,n) \in \ran(f)$ and $r$ is the right endpoint of $I_n$.  We similarly define $f \models \norm{\sum_{j \leq M} \alpha_j \overline{v_j}} > r$.
If $n = \langle \alpha_0, \ldots, \alpha_N \rangle$, then let $\tau_n = \sum_{j \leq N} \alpha_j \overline{v_j}$.  

Let $\Ldisint$ consist of $\Lpres$ together with a family of distinct $0$-ary predicate symbols
$\{S_\nu\}_{\nu \in \N^{<\N}}$ and a family of distinct constants 
$\{\phi_\nu\}_{\nu \in \N^{<\N}}$.    For convenience, and to make the intended meaning clear, write $S_\nu$ as $\nu \in S$.  (We abuse our language; we will use these predicates to mimic subsets of the Baire space.  Although it makes the language heavier, it will allow to compress and unify our formal arguments later in the paper.)  

Suppose $f,g \in \N^\N$.  The intended interpretation of $g$ is as follows. It will encode our current guess on  vectors in the vector subtree. Write $(f,g) \models S_\nu$ if there exists $n$ so that 
$\langle \langle \nu \rangle, n \rangle\in \ran(g)$. 

Let $\tau$ be a term of $\Ldisint$.  Write $\tau$ in the form $\tau_0 + \tau_1$ where 
$\tau_0 = \sum_{j < M} \alpha_j \overline{v_j}$ and $\tau_1 = \sum_{j < M'} \beta_j \phi_{\nu_j}$.  
Write $(f,g) \models \norm{\tau} < r$ if for each $j < M'$ there exist 
$n_j$ so that $\ran(g)$ contains a code of $\{(\nu_j, n_j)\}$ and so that 
\[
f \models \norm{\sum_j \alpha_j \overline{v_j} + \sum_j \beta_j \left(\sum_{j'} \gamma_{j,j'} \overline{v_{j'}}\right)} < r - \sum_j |\beta_j| r_j
\]
where $r_j$ is the radius of the $n_j$-th rational ball of a Banach space presentation and $\langle \gamma_{j,0}, \ldots \rangle$ encodes the center of this ball (again, these objects are independent of the presentation).
We similarly define the meaning of $(f,g) \models \norm{\tau} < r$ (replace subtraction with addition in the above inequality).  

Let $\Lchain$ consist of $\Ldisint$ together with a family of distinct $0$-ary predicate symbols
$\{C_{n, \nu}\}_{n,\nu}$.  For convenience, write $\nu \in C_n$ for $C_{n,\nu}$.  
If $f,g,h \in \N^\N$, write $(f,g,h) \models \nu \in C_n$ if $\ran(h)$ contains a code of 
$\{(n, \nu)\}$.

Suppose $\B^\#$ is a presentation of a Banach space, and let $\sigma$ be a term of $\Ldisint$.  Write $\sigma$ in form 
$\sum_{j < M} \alpha_j \overline{v_j}$.  Let $\sigma[\B^\#] = \sum_{i < M} \alpha_j v_j$ where $v_j$ is the $j$-th distinguished vector of $\B^\#$.
Suppose $\phi$ is a vector tree of $\B^\#$, and let $\tau$ be a term of 
$\Ldisint$.  Write $\tau$ in the form $\sum_{j < M_0} \alpha_j \overline{v_j} + \sum_{j < M_1} \beta_j \phi_{\nu_j}$.  
We then let $\tau[\B^\#, \phi] = \sum_{j <M_0} \alpha_j v_j + \sum_{j < M_1} \beta_j \phi_(\nu_j)$.
The following is an easy consequence of these definitions.

\begin{proposition}\label{prop:interpretation}
Suppose $f$ names a Banach space presentation $\B^\#$ and that $g$ names a vector tree $\psi$ of $\B^\#$.  Let $\sigma$, $\tau$ be terms of $\Lpres$, $\Ldisint$ respectively.
\begin{enumerate}
	\item $f \models \norm{\sigma} < r$ if and only if $\sigma[\B^\#] < r$.
	
	\item $f \models \norm{\sigma} > r$ if and only if $\sigma[\B^\#] > r$.
	
	\item $(f,g) \models \norm{\tau} < r$ if and only if $\tau[\B^\#, \psi] < r$.
	
	\item $(f,g) \models \norm{\tau} > r$ if and only if $\tau[\B^\#, \psi] > r$.	
\end{enumerate}
\end{proposition}

\section{The complexity of naming a disintegration}\label{sec:compl.disint}

In this section, we prove the following which is a stepping stone towards the proof of Main Theorem \ref{mthm:Lp}.

\begin{theorem}\label{thm:disint}
There is a $\Pi_2^0$ predicate $\Disint \subseteq (\N^\N)^3$ so that 
for all $f,g,h \in \N^\N$, if $f$ names a Banach space presentation $\B^\#$, and if $h$ names a real $p \geq 1$, then $\Disint(f,g,h)$ if and only if $g$ names a formal $L^p$-disintegration of 
$\B^\#$.
\end{theorem}

The proof of this fact is essentially reduced to a careful analysis of the definition of a formal $L^p$-disintegration. However, since we are dealing with 
separable spaces rather than countable discrete algebras, a brute-force quantifier counting would not suffice, for only positive existential formulae and their (infinite) disjunction will correspond to c.e.~facts. However, in our case with some care the complexity will be equal to the natural (`naive') estimate, but proving this requires some care. We first prove two technical lemmas.

\begin{lemma}\label{lm:vectortree}
There is a $\Pi_2^0$-predicate $\VectorTree \subseteq \N^\N \times \N^\N$ so that 
whenever $f \in \N^\N$ names a Banach space presentation $\B^\#$, 
$\VectorTree(f,g)$ if and only if $g$ names a vector tree of $\B^\#$.
\end{lemma}

\begin{proof}
We first claim that there is a $\Sigma_1^0$ predicate $\IncludesClosure \subseteq \N^\N \times \N^2$ so that 
whenever $f$ is a name of a Banach space presentation $\B^\#$ and $n_0, n_1 \in \N$,  
$\IncludesClosure(f; n_0, n_1)$ if and only if $B(\B^\#; n_0) \supseteq \overline{B}(\B^\#; n_1)$.  This is immediate from the following observation:
if $u$ and $v$ are vectors of a Banach space $\B$, and if $r$ and $s$ are positive rational numbers, then $\overline{B}(u;r) \subseteq B(v;r)$ if and only if 
$\norm{u - v}_\B + r < s$.  So, by Proposition \ref{prop:interpretation}, we may define $\IncludesClosure(f;n_0, n_1)$ to hold if 
\[
f \models \norm{\sum_{j < M_0} \alpha_{0,j} \overline{v_j} - \sum_{j < M_1} \alpha_{1,j} \overline{v_j} } < r_1 - r_0
\]
where $\langle \alpha_{k,0}, \ldots \rangle$ encodes the center of the $n_k$-th rational ball and $r_k$ is the radius of this ball.

When $g \in \N^\N$, let $S_g$ denote the set of all $\nu \in \N^{< \N}$ so that $g$ contains 
a code of $\{(\nu, n)\}$ for some $n$.  When $g \in \N^\N$, and when 
$\nu \in S_g$, let $S_{g,\nu}$ denote the set of all $n \in \N$ so that 
$\ran(g)$ contains a code of $\{(\nu, n)\}$.

\newcommand{\FD}{\operatorname{FD}}
Let us say that two rational balls are \emph{formally disjoint} if the distance between their centers 
is larger than the sum of their radii.  It follows that there is a $\Sigma_1^0$ predicate $\FD \subseteq \N^\N \times \N^2$ so that whenever $f$ names a Banach space presentation $\B^\#$, 
$\FD(f;m,n)$ if and only if $B(\B^\#; m)$ and $B(\B^\#; n)$ are formally disjoint.

Let $\VectorTree(f,g)$ if and only if all of the following hold.
\begin{enumerate}
	\item $S_g$ is a tree.\label{lm:vectortree::cond.1}
	
	\item For every $\nu \in S_g$ and every $k \in \N$, there is an $n \in S_{g,\nu}$ so that 
	the radius of the $n$-th rational ball is at most $2^{-k}$. \label{lm:vectortree::cond.2}
	
	\item For all $\nu \in S_g$ and all $n,n' \in S_{g,\nu}$, there exists $m \in S_{g,\nu}$ so that 
	$\IncludesClosure(f;m,n)$ and $\IncludesClosure(f;m,n')$.  \label{lm:vectortree::cond.3}
	
	\item For all $\nu \in S_g$ and all $m,n \in \N$, if $m \in S_{g, \nu}$, and if 
	$\IncludesClosure(f;m,n)$, then 
	$n \in S_{g,\nu}$. \label{lm:vectortree::cond.4}
	
	\item For all distinct $\nu_0, \nu_1 \in S_g$, there exist $n_0 \in S_{g,n_0}$ and $n_1 \in S_{g,n_1}$ so that $\FD(f;n_0,n_1)$.  \label{lm:vectortree::cond.5}	
\end{enumerate} 
Since $\IncludesClosure$ and $\FD$ are $\Sigma_1^0$, it follows that $\VectorTree$ is $\Pi_2^0$.  

Suppose $f$ names $\B^\#$.  Let $g \in \N^\N$.  If $g$ names a vector tree of $\B^\#$, then 
it is routine to verify that $\VectorTree(f,g)$ holds.  
So, suppose $\VectorTree(f,g)$ holds.  
Let $S = S_g$.  By (\ref{lm:vectortree::cond.2}), for each $k \in \N$ and each $\nu \in S$, there is an 
$n_k \in S_{g,\nu}$ so that the radius of $B(\B^\#;n_k)$ is at most $2^{-k}$.  Let $c_k$ denote the 
center of $B(\B^\#;n_k)$.  It follows from (\ref{lm:vectortree::cond.3}) that 
$\{c_k\}_k$ is a Cauchy sequence; let $\psi(\nu)$ denote its limit.  

We now show that 
for each $\nu \in S$, $S_{g,\nu} = \{n\ :\ \psi(\nu) \in B(\B^\#; n)\}$.  Suppose $n \in S_{g,\nu}$.  
By (\ref{lm:vectortree::cond.3}), there is an $m \in S_{g,\nu}$ so that 
$\overline{B(\B^\#; m)} \subseteq B(\B^\#; n)$.  
By (\ref{lm:vectortree::cond.3}) again, $B(\B^\#;n_k) \cap \overline{B(\B^\#; m)} \neq \emptyset$
for each $k$.  Thus, $\psi(\nu) \in \overline{B(\B^\#; m)}$.  
Conversely, suppose $\psi(\nu) \in B(\B^\#; n)$.  There exists $k$ so that 
$\overline{B(\B^\#;n_k)} \subseteq B(\B^\#; n)$.  Hence, by (\ref{lm:vectortree::cond.4}), 
$n \in S_{g, \nu}$.

It follows from (\ref{lm:vectortree::cond.5}) that $\psi$ is injective.  Thus, $\psi$ is a vector tree.
\end{proof}

\begin{lemma}\label{lm:summative}
Suppose $\psi$ is an $L^p$-formally separating vector tree of $\B$.  Suppose also that for every nonterminal $\nu \in \dom(\psi)$ and finite set $F \subseteq \dom(\psi)$ of children of $\nu$, $\sum_{\mu \in F} \psi(\mu)$ is an $L^p$-formal component of $\psi(\nu)$.  
Then, $\psi$ is summative if and only if for every $\epsilon > 0$ and every nonterminal $\nu \in \dom(\psi)$, there is a finite set $F \subseteq \dom(\psi)$ of children of $\nu$ so that 
\begin{equation}
\norm{\psi(\nu) - \sum_{\mu \in F} \psi(\mu)}_\B < \epsilon. \label{ineq:sum}
\end{equation}
\end{lemma}

\begin{proof}
Let $S = \dom(\psi)$.  
The first direction is trivial.  Let $\epsilon > 0$, and choose a finite set $F \subseteq S$ of children of $\nu$ so that (\ref{ineq:sum}) holds.  Suppose $F'$ is a finite set of children of $\nu$ so that 
$F \subseteq F' \subseteq S$.  We claim that 
\begin{equation}
\norm{\psi(\nu) - \sum_{\mu \in F} \psi(\mu)}_\B \geq \norm{\psi(\nu) - \sum_{\mu \in F'} \psi(\mu)}_\B.\label{ineq:F}
\end{equation}
For, since $\sum_{\mu \in F'} \psi(\mu)$ is an $L^p$-formal component of $\psi(\nu)$, 
\[
\norm{\psi(\nu) - \sum_{\mu \in F'} \psi(\mu)}_\B^p = \norm{\psi(\nu)}_\B^p - \norm{\sum_{\mu \in F'} \psi(\mu)}_\B^p.
\]
Since $\psi$ is $L^p$-formally separating and $F \subseteq F'$, 
\[
\norm{\sum_{\mu \in F'} \psi(\mu)}_\B^p = \norm{\sum_{\mu \in F} \psi(\mu)}_\B^p + \norm{\sum_{\mu \in F' - F} \psi(\mu)}_\B^p.
\]
Thus, (\ref{ineq:F}).

It follows that $\psi(\nu) = \sum_\mu \psi(\mu)$ where $\mu$ ranges over the children of $\nu$ in $S$.
\end{proof}

\begin{proof}[Proof of Theorem \ref{thm:disint}]
Let $f,g \in \N^\N$, and let $\tau_0, \ldots, \tau_n$ be terms of $\Ldisint$.  
Let us say that $(f,g)$ \it ensures 
$\norm{\ }$ is $p$-additive on $(\tau_0, \ldots, \tau_n)$\rm if 
there do not exist rational numbers $r_0, \ldots, r_{n+1}$ and $m \in \N$ so that one of the following holds.
\begin{enumerate}
	\item $(f,g) \models r_{n+1} > \norm{\tau_0 + \ldots + \tau_n}$, 
	$(f,g) \models r_j < \norm{\tau_j}$, and 
	$r_{n+1}^p < r_0^p + \ldots + r_n^p$.  
	
	\item $(f,g) \models r_{n+1} < \norm{\tau_0 + \ldots + \tau_n}$, 
	$(f,g) \models r_j >\norm{\tau_j}$, and 
	$r_{n+1}^p > r_0^p + \ldots + r_n^p$. 
\end{enumerate}

Let us say that $(f,g,h)$ ensures $\phi$ is \emph{formally separating} if for all rational scalars $\alpha_0, \ldots, \alpha_n$ and all pairwise incomparable
$\nu_0, \ldots, \nu_n \in \N^{< \N}$ so that $(f,g) \models \nu_j \in S$ for each $j$, 
$(f,g,h)$ ensures 
$\norm{\ }$ has the additivity of the real denoted by $h$ on $(\alpha_0\phi_{\nu_0}, \ldots, \alpha_n\phi_{\nu_n})$.

Suppose $f$ names $\B^\#$ and $g$ names a vector tree $\psi$ of $\B^\#$. 
It follows from Proposition \ref{prop:interpretation} and a simple continuity argument that 
$(f,g,h)$ ensures $\norm{\tau_0 + \ldots + \tau_n}$ is $p$-additive if and only if 
\[
\norm{\tau_0[\B^\#, \psi] + \ldots + \tau_n[\B^\#, \psi]}_{\B}^p = 
\norm{\tau_0[\B^\#,\psi]}_{\B}^p + \ldots + \norm{\tau_n[\B^\#, \psi]}_{\B}^p.
\] 
If $h$ names a real $p$, it then follows that $(f,g,h)$ ensures $\phi$ is formally separating if and only if $\psi$ is formally $L^p$-separating.  

So, let $\Disint(f,g,h)$ if and only if the following hold.  
\begin{enumerate}
	\item $\VectorTree(f,g)$.
	
	\item $(f,g,h)$ ensures $\phi$ is formally separating.  
	
	\item For all $\nu_0, \ldots, \nu_n \in \N^{< \N}$ and all rational scalars $\alpha, \beta$, if 
	$(f,g) \models \nu_j \in S$ for each $j$, and if 
	$\nu_1, \ldots, \nu_n$ are distinct children of $\nu_0$, then $(f,g,h)$ ensures 
	$\norm{\alpha (\phi(\nu_0) - \sum_{1 \leq j \leq n} \phi(\nu_j)) + \beta \sum_{1 \leq j \leq n} \phi(\nu_j)}$ has the additivity of the real denoted by $h$.
	
	\item Whenever $k \in \N$ and $\nu$ is a node so that $(f,g,h) \models \nu' \in S$ for at least 
	one $\nu' \supset \nu$, there exist distinct children $\nu_0, \ldots, \nu_n$ of $\nu$ so that 
	$(f,g) \models \nu_j \in S$ and so that $(f,g) \models \norm{\phi(\nu) - \sum_{j \leq n} \phi(\nu_j)} < 2^{-k}$.
	
	\item For every $j,k \in \N$, there exist $\nu_0, \ldots, \nu_n \in \N^{< \N}$ and rational scalars 
	$\alpha_0, \ldots, \alpha_n$ so that 
	$(f,g) \models \nu_m \in S$ for each $m$ and so that $(f,g) \models \norm{v_j - \sum_{m \leq n} \alpha_m \phi(\nu_m)} < 2^{-k}$.  
\end{enumerate}

Suppose $f$ names $\B^\#$ and $h$ names a real $p \geq 1$.  
If $g$ names a formal $L^p$-disintegration of $\B^\#$, then it is routine to verify $\Disint(f,g,h)$.  
So, suppose $\Disint(f,g,h)$.  Thus, $g$ names a vector tree $\psi$.  Let $S = \dom(\psi)$.  By what has just been observed, 
$\psi$ is formally $L^p$-separating.  
It also follows that $\B$, $p$, $\psi$ satisfy the hypotheses of Lemma \ref{lm:summative}, and so $\psi$ is summative.  Finally, the last condition of the definition of 
$\Disint$ ensures that the range of $\psi$ is linearly dense.
\end{proof}

\section{Index sets of computable presentations of Lebesgue spaces with known exponent}\label{sec:index.exp.known}

The goal of this section is to prove the following three theorems.

\begin{theorem}\label{thm:Lp}
Suppose $p \geq 1$ is a computable real.  Then, the set of all names of $L^p$ space presentations 
is $\Pi_2^0$-complete.
\end{theorem}

\begin{theorem}\label{thm:L2}
The set of all names of presentations of $\ell^2$ is $\Pi_2^0$-complete as is the set of all names of 
presentations of $\ell^2_n$ for each $n \geq 1$.  
\end{theorem}

\begin{theorem}\label{thm:known.exp} Let $p$ be a computable real so that $p \geq 1$ and $p \neq 2$.  
\begin{enumerate}
	\item For each $n \geq 1$, the set of all names of presentations of $\ell^p_n$ is 
	$\Pi_2^0$-complete.\label{thm:known.exp.lpn}
	
	\item The set of all names of presentations of $\ell^p$ is $\Pi_3^0$-complete.\label{thm:known.exp.lp}
	
	\item The set of all names of presentations of $L^p[0,1]$ is $\Pi_2^0$-complete.\label{thm:known.exp.Lp01}
	
	\item For each $n \geq 1$, the set of all names of presentations of $\ell^p_n \oplus_p L^p[0,1]$ is 
	$d$-$\Sigma_2^0$-complete.\label{thm:known.exp.lpnLp01}
	
	\item The set of all names of presentations of $\ell^p \oplus_p L^p[0,1]$ is 
	$d$-$\Sigma_3^0$-complete.\label{thm:known.exp.lpLp01}
\end{enumerate}
\end{theorem}

The upper bounds in these theorems are mostly obtained by carefully using the technology developed in previous sections.  When the upper bound is $\Pi^0_2$, the lower bound follows from Lemma \ref{lm:Banach.complete}.  In other cases, lower bounds are obtained by relating the complexity of $L^p$-space index sets to that of certain classes of countable linear orders.  To this end, we demonstrate in Proposition \ref{prop:transfer} the existence of a computable operator that given a diagram $f$ of a linear order $L$ produces a name for an $L^p$-space $L^p(\Omega)$ such that the number of adjacencies in $L$ is the number of atoms in $\Omega$, and so that $L^p[0,1]$ isometrically embeds into $L^p(\Omega)$ if $L$ contains a copy of $\Q$ (c.f. Proposition \ref{prop:transfer}).  

We begin by proving two lemmas for the sake of proving the upper bound in Theorem \ref{thm:Lp}.

\begin{lemma}\label{lm:hilbert}
There is a $\Pi_1^0$ predicate $\Hilbert$ so that whenever $f$ names a Banach space presentation $\B^\#$, 
$\B$ is a Hilbert space if and only if $\Hilbert(f)$.
\end{lemma}

\begin{proof} 
Suppose $f$ names a Banach space presentation $\B^\#$.  Then, $\Hilbert(f)$ if and only if 
the $\B$ satisfies the parallelogram law, which is clearly an (effectively) closed condition.  
More formally,
let $\Hilbert(f)$ hold if there do not exist $r_0,r_1, r_2, r_3 \in \Q$ so that one of the following 
holds.
\begin{enumerate}
	\item $f \models r_0 > \norm{\tau_0 + \tau_1}$,\ \  $f \models r_1>  \norm{\tau_0 - \tau_1}$,\ \  
	$f \models r_2 < \norm{\tau_0}$,\ \ $f \models r_3 < \norm{\tau_1}$, and 
	$2(r_2^2 + r_3^2) > r_0^2 + r_1^2$.
	
	\item $f \models r_0 < \norm{\tau_0 + \tau_1}$,\ \  $f \models r_1 < \norm{\tau_0 - \tau_1}$,\ \  
	$f \models r_2 > \norm{\tau_0}$,\ \  $f \models r_3 > \norm{\tau_1}$, and 
	$2(r_2^2 + r_3^2) < r_0^2 + r_1^2$.
\end{enumerate} \end{proof}

\begin{lemma}\label{lm:Lspace}
There is a $\Pi_2^0$ predicate $\Lspace \subseteq \N^\N \times \N^\N$ so that 
for all $f,g \in \N^\N$, if $g$ names a real $p \geq 1$, then $\Lspace(f,g)$ if and only if $f$ names an $L^p$-space presentation.  
\end{lemma}

\begin{proof} 
Recall the predicate $\Disint$ from Theorem~\ref{thm:disint}, and recall that 
 $\Phidis : (\N^\N)^2 \rightarrow \N^\N$ is a computable operator so that 
for all $f,g \in \N^\N$, if $g$ names a real $p \geq 1$ so that $p \neq 2$, and if 
$f$ names an $L^p$ space presentation $\B^\#$, then $\Phidis(f,g)$ names 
a disintegration of $\B^\#$.

Let $\Lspace(f,g)$ hold if and only if $\Banach(f)$ and 
\[
(\Hilbert(f)\ \wedge\ \mbox{$g$ names $2$})\ \vee\ 
\Disint(f, \Phidis(f,g), g).
\]
The set of all names of $2$ is $\Pi_1^0$.  So, it follows from Theorem \ref{thm:Banach.complete}, Lemma \ref{lm:hilbert}, and Theorem \ref{thm:disint} that $\Lspace$ is $\Pi_2^0$.  
Assume $g$ names a real $p \geq 1$.  
Then, $\Lspace(f,g)$ says that either $f$ names a Hilbert space and $p=2$ or it names a Banach space presentation with an $L^p$-formal disintegration.  By Theorem \ref{thm:disint.Lp}, the latter holds if and only if $f$ names an $L^p$-space presentation.
\end{proof}

\begin{proof}[Proof of Theorem \ref{thm:Lp}]  
In light of Lemma \ref{lm:hilbert} and \ref{lm:Lspace}, it suffices to show $\Pi_2^0$-hardness.
The case where $p = 2$, follows from Theorem~\ref{thm:Banach.complete}. 
%We do not really need p=2 in this thm, because it uses nothing but the totality of the norm. I am not sure why we need the stuff below.
 So, assume $p \neq 2$.  Let $Q \subseteq \N^\N \times \N^2$ be a computable predicate so that for all 
$f \in \N^\N$, 
\[
P(f) \Leftrightarrow \forall x \exists y Q(f; x,y).
\]
For each $f \in \N^\N$ we define a Banach space presentation $\B_f^\#$ as follows.  
When $f \in \N^\N$ and $x,y \in \N$, let:
\begin{eqnarray*}
R_f(\langle x, 0 \rangle) & = & e_{3x} + e_{3x + 1} \\
R_f(\langle x,1 \rangle) & = & e_{3x+1} + e_{3x+2}\\
R_f(\langle x,y +2 \rangle) & = & 
\left\{\begin{array}{cc}
e_{3x+1} & \exists y' < y\ Q(f; x,y')\\
e_{3x+1} + e_{3x+2} & \mbox{otherwise} \\
\end{array}
\right.
\end{eqnarray*}
Let $\B_f$ denote the closed linear span of $R_f$ in $\ell^p$, and let 
$\B_f^\# = (\B_f, R_f)$.  

We now define $F$.  Since $Q,p$ are computable, $R_f$ is an $f$-computable sequence of 
$\ell^p$ uniformly in $f$.  Thus, $\B_f^\#$ is an $f$-computable presentation uniformly in $f$.   
Hence, %by Proposition \ref{prop:operator.enumerator}, 
there is a computable $F : \N^\N \rightarrow \N^\N$ so that $F(f)$ names 
$\B_f^\#$ for each $f \in \N^\N$.  

We now show that for all $f \in \N^\N$, $P(f)$ if and only if $F(f)$ enumerates the diagram of an $L^p$-space presentation.  If $P(f)$, then the definition of $R_f$ ensures $\B_f = \ell^p$.  
Suppose $P(f)$ fails, and by way of contradiction suppose $\B_f$ is an $L^p$ space.

We claim $\B_f$ is isometrically isomorphic to $\ell^p$.  As the vectors $R_f(\langle 0,0 \rangle)$, $R_f(\langle 1,0 \rangle)$, $R_f(\langle 2,0 \rangle )$ $\ldots$ are linearly independent, $\B_f$ is infinite-dimensional.   By the classification of separable $L^p$ spaces, $\B_f$ is isometrically isomorphic to one of $\ell^p$, $L^p[0,1]$, $\ell^p \oplus_p L^p[0,1]$, or $\ell^p_n \oplus_p L^p[0,1]$ for some $n \geq 1$.  
Since $\B_f$ is a subspace of $\ell^p$, it must be that $\B_f$ is isometrically isomorphic to $\ell^p$.

Let $T$ be an isometric isomorphism of $\B_f$ onto $\ell^p$, and let $v_j = T^{-1}(e_j)$.  
Thus, $\{v_j\}_{j \in \N}$ is a Schauder basis for $\B_f$.  Furthermore, by Theorem \ref{thm:disj.supp}, the vectors $v_0$, $v_1$, $\ldots$, are disjointly supported.  

Since $P(f)$ fails, there is an $x_0 \in \N$ so that $Q(f; x_0, y)$ fails for all $y \in \N$.  We claim there is a $j_0 \in \N$ so that $e_{3x_0} + e_{3x_0 + 1}$ is a scalar multiple of $v_{j_0}$.  
For there is a sequence $\{\beta_j\}_{j \in \N}$ of scalars so that $e_{3x_0} + e_{3x_0 + 1} = \sum_j \beta_j v_j$.  Thus, $\beta_j v_j$ is a component of $e_{3x_0} + e_{3x_0 + 1}$ for each $j$.  
There is a $j_0 \in \N$ so that $\beta_{j_0} \neq \zerovec$.  Hence, $v_{j_0}$ is either $e_{3x_0}$, 
$e_{3x_0 + 1}$, or $e_{3x_0} + e_{3x_0 + 1}$.  Since $Q(f;x_0, y)$ fails for all $y \in \N$, 
$e_{3x_0 + 1} \not \in \B_f$.  Hence, $e_{3x_0} \not \in \B_f$ either.  Thus, 
$\beta_{j_0} v_{j_0} = e_{3x_0} + e_{3x_0 + 1}$.  

It similarly follows that there is a $j_1 \in \N$ so that $v_{j_1}$ is a scalar multiple of 
$e_{3x_0 + 1} + e_{3x_0 + 2}$.  Hence $j_0 \neq j_1$, and $v_{j_0}$, $v_{j_1}$ are not disjointly supported.  This is a contradiction, and so $\B_f$ is not an $L^p$ space.
\end{proof}

\begin{proof}[Proof of Theorem \ref{thm:L2}]
Recall that if $\F = \R$, then the inner product of an inner product space and its norm are related by the polarization identity
\[
\langle u,v \rangle = \frac{1}{4} (\norm{u + v}^2 - \norm{u - v}^2).
\]
If $\F = \C$, then the corresponding identity is 
\[
\langle u,v \rangle = \frac{1}{4} (\norm{u + v}^2 - \norm{u - v}^2) + \frac{i}{4}(\norm{u - iv}^2 - \norm{u + iv}^2).
\]
In either case, the point is that the inner product can be expressed in terms of the norm.  
Recall also that if $v_0, v_1, \ldots, v_n$ are pairwise orthoganol unit vectors, and if 
$v = \alpha_0 v_0 + \ldots + \alpha_n v_n$, then $\alpha_j = \langle v, v_j \rangle$ for each $j$.  
It follows from these observations and Gramm-Schmidt orthonormalization that there is a computable
operator $F : \subseteq \N^\N \rightarrow \N$ so that for all $f \in \N^\N$, if $f$ names a nonzero Hilbert space presentation $\B^\#$, then $F(f)$ enumerates an orthonormal basis for $\B^\#$.  If $\B$ is finite-dimensional, then this enumeration will contain repetitions.  
It then follows that there is a $\Sigma_1^0$ predicate $P \subseteq \N^\N \times \N$ so that for all 
$f \in \N^\N$ and $n \in \N$, if $f$ names a Hilbert space presentation $\B^\#$, then 
$P(f;n)$ if and only if the dimension of $\B$ is at least $n$.  Thus, $f \in \N^\N$ names 
$\ell^2$ if and only if 
\[
\Banach(f)\ \wedge\ \Hilbert(f)\ \wedge\ \forall n\ P(f;n).
\]
Also, $f \in \N^\N$ names $\ell^2_n$ if and only if 
\[
\Banach(f)\ \wedge\ \Hilbert(f)\ \wedge\ P(f;n)\ \wedge\ \neg P(f;n+1).
\]
Both of these conditions are $\Pi_2^0$.  

Again, the lower bounds are by Lemma \ref{lm:Banach.complete}.   
\end{proof}

\begin{proof}[Proof of Theorem \ref{thm:known.exp}.\ref{thm:known.exp.lpn}]
Fix a computable name $h$ of $p$. 

We first show that there is a $d-\Sigma_1^0$ predicate $G \subseteq \N^\N$ so that whenever $f \in \N^\N$ names an $L^p$ space presentation $\B^\#$, $G(f)$ if and only if $\B$ is isometrically isomorphic to $\ell^p_n$.  To this end, we first note that there is a $\Sigma_1^0$-predicate 
$P \subseteq \N^\N \times \N^\N \times \N$ so that for all $f,g \in \N^\N$ and $m \in \N$, if 
$f$ names an $L^p$ space presentation $\B^\#$, and if $g$ names a disintegration $\psi$ of $\B^\#$, then $P(f,g;m)$ if and only if $\dom(\psi)$ contains an antichain of size $m$.  Let 
\[
G(f) \Leftrightarrow P(f, \Phidisint(f,h); n)\ \wedge\ \forall m > n \neg P(f, \Phidisint(f,h); m).
\] 
Thus, $G$ is $d$-$\Sigma_1^0$.  It follows from Lemma \ref{lm:antichains} that $G$ has the required properties.

We now infer from Lemma \ref{lm:Lspace} that the set of all names of presentations of $\ell^p_n$ is $\Pi_2^0$.  Once again, the lower bound follows from Lemma \ref{lm:Banach.complete}.
\end{proof}

The lower bounds in the remaining parts of Theorem \ref{thm:known.exp} require the linear order technology alluded to in the introduction of this section and which we now lay out precisely.

\begin{definition}\label{def:interval.algebras}
Suppose $L$ is a set of reals that contains at least two points.  
\begin{enumerate}
	\item Let $I_L$ denote the set of all open intervals whose endpoints belong to $L$.
	
	\item Let $X_L = \bigcup I_L$.
	
	\item Let $\mathcal{S}_L$ denote the $\sigma$-algebra over $X_L$ generated by $I_L$.  
	
	\item Let $\Omega_L = (X_L, \mathcal{S}_L, m_L)$ where $m_L$ is the restriction of Lebesgue measure to $\mathcal{S}_L$.
\end{enumerate}
\end{definition}

\begin{lemma}\label{lm:omega.L}
Suppose $L$ is a set of reals that contains at least two points.  
\begin{enumerate}
	\item Every minimal element of $I_{\overline{L}}$ is an atom of $\Omega_L$, and every atom of $\Omega_L$ is, up to a set of measure 0, a minimal element of $I_{\overline{L}}$.  \label{lm:omega.L::atom}
	
	\item If there is an open interval $I$ so that $I \cap \Q \subseteq L$, then $\Omega_L$ is not purely atomic.\label{lm:omega.L::non.atom}
\end{enumerate}
\end{lemma}

\begin{proof}
Suppose $(a,b)$ is a minimal element of $I_{\overline{L}}$.  Let $\mathcal{S}$ denote the set of all 
$A \in \mathcal{S}_L$ so that $A$ either includes $(a,b)$ or is disjoint from $(a,b)$.  
Then, $I_L \subseteq \mathcal{S}$, and $\mathcal{S}$ is a $\sigma$-algebra.  
Hence, $\mathcal{S} \supseteq \mathcal{S}_L$.  Since $a,b \in \overline{L}$, 
$(a,b) \in \mathcal{S}$.  Therefore, $(a,b)$ is an atom of $\Omega_L$.

Conversely, suppose $A$ is an atom of $\Omega_L$.  Since $I_L$ generates $\Omega_L$, for each $n \in \N$, there is an interval $J_n$ in $I_L$ so that $\mu(J_n \triangle A) < \min\{\mu(A), 2^{-n}\}$ (see e.g. Theorem A p. 168 of \cite{Halmos.1950}).  Thus, $A - J_n$ is null for each $n$.  
Let $J_n = (a_n, b_n)$.  Without loss of generality, we may assume $J_{n+1} \subseteq J_n$.  
Let $a = \lim_n a_n$, and let $b = \lim_n b_n$.  Thus, $\mu((a,b) \triangle A) = 0$, and 
$a,b \in \overline{L}$.  

Part (\ref{lm:omega.L::non.atom}) follows from part (\ref{lm:omega.L::atom}).
\end{proof}

\begin{definition}\label{def:faithful}
Suppose $L$ is a countable linear order, and suppose $F : L \rightarrow \R$ is an order monomorphism. 
We say that $F$ is \emph{faithful} if every adjacency of $\overline{\ran(F)}$ is an adjacency of 
$\ran(F)$.  
\end{definition}

\begin{corollary}\label{cor:faithful}
Suppose $L$ is a countable linear order, and suppose $F : L \rightarrow \R$ is a faithful order isomorphism.
Then, the number of atoms of $\Omega_{\ran(F)}$ is the number of adjacencies of $L$.
\end{corollary}

\begin{lemma}\label{lm:faithful}
There is a computable operator $G : 2^\N \times \N \rightarrow \Q$ so that whenever $f$ is the diagram 
of a linear order $L$, the map $G(f;\cdot)$ is a faithful embedding of 
$L$ into $\R$.
\end{lemma}

\begin{proof}
Let $f \in 2^\N$.  Set $G(f;0) = 0$.  Define $G(f; s + 1)$ as follows.  Set:
\begin{eqnarray*}
M & = & \max\{G(f;t)\ :\ t \leq s\}\\
m & = & \min\{G(f;t)\ :\ t \leq s\}.
\end{eqnarray*}
If $f(\Quinequote{t \leq s+1}) = 1$ for all $t \leq s$, then set $G(f; s + 1) = M + 1$.  
If $f(\Quinequote{t \leq s+1}) = 0$ for all $t \leq s$, then set $G(f; s + 1) = m - 1$. 
Suppose neither of these two cases holds.  Set:
\begin{eqnarray*}
m_0 & = & \min\{G(f; t)\ :\ f(\Quinequote{ s +1 \leq t}) = 1\}\\
M_0 & = & \max\{G(f; t)\ :\ f(\Quinequote{ t \leq s +1}) = 1\}\\
G(f;s+1) & = & \frac{1}{2}(m_0 + M_0)
\end{eqnarray*}
Suppose $f$ is the diagram of a linear order $L = (\N, \leq_0)$.  
Let $F(t) = G(f;t)$.  
We show $F$ is faithful.  Let 
$\{a,b\}$ be an adjacency of $\overline{\ran(F)}$.  Without loss of generality, assume $a < b$.  
By way of contradiction, suppose one of $a,b$ does not belong of $\ran(F)$.  
We consider the case where neither belongs to $\ran(F)$; the other cases are handled similarly.   
Since $a,b$ are boundary points of $\ran(F)$, there exist increasing sequences 
$\{s_j\}_{j \in \N}$ and $\{t_j\}_{j \in \N}$ so that $\{F(s_j)\}_{j \in \N}$ is increasing, 
$\{F(t_j)\}_{j \in \N}$ is decreasing, $\lim_j F(s_j) = a$, and $\lim_j F(t_j) = b$.  There is a $\delta > 0$
so that whenever $(x,y) \in (a - \delta, a) \times (b, b + \delta)$, $\frac{1}{2}(x + y) \in (a,b)$.  
There is a $j_0 \in \N$ so that $F(s_j) > a - \delta$ and $F(t_j) < b + \delta$ whenever $j \geq j_0$.
It then follows from the definition of $G$ that $F(s_{j_0 + 1}) \in (a,b)$- a contradiction.
\end{proof}

\begin{proposition}\label{prop:transfer}
Suppose $p \geq 1$ is computable.   Then, there is a computable operator $F : 2^\N \rightarrow \N^\N$ so that for all $f \in 2^\N$, if $f$ is the diagram of a linear order $L$, then $F(f)$ names a presentation $L^p(\Omega)^\#$ so that the number of atoms of $\Omega$ is the number of adjacencies of $L$ and so that $L^p[0,1]$ isometrically embeds into $L^p(\Omega)$ if 
there is an interval of $L$ that is isomorphic to $\Q$.  
\end{proposition}

\begin{proof}
Let $G$ be a computable operator as in Lemma \ref{lm:faithful}.  Let $S$ be a structure on 
$L^p(\R)$ so that $\{S(n)\}_n$ enumerates all rational simple functions and so that 
$(L^p(\R), S)$ is computable.  Without loss of generality, we assume $S(0) = \zerovec$.  
There is a total computable function $R : 2^\N\times \N \rightarrow \N$ so that for all $f \in 2^\N$, $S \circ R(f)$ enumerates
\[
\{\zerovec\} \cup \{\Ind_{(a,b)}\ :\ a,b \in \ran(G(f))\ \wedge\ a < b\}.
\] 
Let $H(f; n) = S(R(f;n))$.  Let $\mathcal{B}_f$ denote the closed linear span of $\ran(H(f))$.  
Thus, if $f$ is the diagram of a linear order $L$, then $B_f$ is 
$L^p(\Omega_L)$.  Since $G$ and $(L^p(\R), S)$ are computable, there is a computable operator 
$F: 2^\N \rightarrow \N^\N$ so that $F(f)$ enumerates the diagram of $(\B_f, H(f))$.  By Lemma 
\ref{lm:omega.L}, $F$ has the required properties.
\end{proof}

\begin{lemma}\label{lm:Lp01.embed}
Suppose $p$ is a computable real so that $p \geq 1$ and $p \neq 2$.  There is a $\Sigma_3^0$ predicate $E \subseteq \N^\N$ so that whenever $f$ names an $L^p$ space presentation $\B^\#$, $E(f)$ if and only if $L^p[0,1]$ isometrically embeds into $\B$.
\end{lemma}

\begin{proof}
Suppose $\psi$ is a disintegration of an $L^p$ space $\B$, and let $\{C_n\}_{n < \kappa}$ be a decomposition of $\dom(\psi)$ into almost norm-maximizing chains.  For each $n < \kappa$, 
let $g_n$ denote the $\preceq$-infimum of $\phi[C_n]$.  Then, by Lemma 3.4 of \cite{Brown.McNicholl.2018}, 
$L^p[0,1]$ embeds into $\B$ if and only if $\psi(\emptyset) \neq \sum_n g_n$.  Since $g_0$, $g_1$, $\ldots$ are disjointly supported components of $\psi(\emptyset)$, it follows that $\norm{\psi(\emptyset) - \sum_n g_n}^p = \norm{\psi(\emptyset)}^p - \sum_n \norm{g_n}^p$.  Thus: 
\begin{eqnarray*}
\psi(\emptyset) \neq \sum_n g_n & \Leftrightarrow & \exists k \forall n < \kappa\ \sum_{j = 0}^n \norm{g_n}_\B^p < \norm{\psi(\emptyset)}_\B^p - 2^{-k}\\
& \Leftrightarrow & \exists k \forall n < \kappa \exists \nu_0, \ldots, \nu_n\ [(\forall j \leq n\ \nu_j \in C_j) \Rightarrow\ \sum_{j \leq n} \norm{\psi(\nu_j)}_\B^p < \norm{\psi(\emptyset)}_\B^p - 2^{-k}].
\end{eqnarray*}
It now follows via the technology of Proposition \ref{prop:interpretation} that there is a $\Sigma_3^0$ predicate $P \subseteq (\N^\N)^3$ so that for all $f,g,h \in \ N^\N$, if $f$ names an $L^p$ space presentation $\B^\#$, and if $g$ names a disintegration $\psi$ of $\B^\#$, and if $h$ names a decomposition $\{C_n\}_{n < \kappa}$ of $\dom(\psi)$ into almost norm-maximizing chains, then $P(f,g,h)$ if and only if $\psi(\emptyset) \neq \sum_{n < \kappa} g_n$ where for each $n$ $g_n$ denotes the $\preceq$-infimum of $\phi[C_n]$.  The existence of $E$ now follows from the uniformity of Theorems \ref{thm:disint.comp} and \ref{thm:decomp}.
\end{proof}

\begin{lemma}\label{lm:sigma.03}
Suppose $P \subseteq \N^\N$ is $\Sigma_3^0$.  Then, there is a computable operator 
$F : \N^\N \rightarrow 2^\N$ so that for all $f \in \N^\N$, $F(f)$ is the diagram of a linear order $L_f$ so that for some $n \geq 1$, 
\[
L_f \approx \left\{
\begin{array}{cc} 
\eta + n & \mbox{if $P(f)$}\\
\omega & \mbox{otherwise.}\\
\end{array}
\right.
\]
Furthermore, if $A \subseteq \N$ is infinite and computable, then 
we can ensure that $n$ has the form $2^{s_1} + \ldots + 2^{s_m}$ where $s_1, \ldots, s_m \in A$ are distinct.
\end{lemma}

\begin{proof}
Let $Q \subseteq \N^\N \times \N^3$ be a computable predicate so that for all $f \in \N^\N$, 
\[
P(f) \Leftrightarrow \exists x \forall y \exists z\ Q(f; x,y,z).
\]
Let 
\[
\ell(f;x,s) = \max(\{y \leq s\ :\ \forall y' < y \exists z \leq s\ Q(f; x,y',z)\}\ \cup\ \{0\}).
\]
If for every $x$ there is an $s' < s$ so that $\ell(f; x,s) \leq \ell(f; x, s')$, then let 
$\delta(f; s) = \infty$.  Otherwise, let $\delta(f;s)$ denote the least $x$ so that 
$\ell(f;x,s) > \ell(f;x,s')$ for all $s' < s$.  
Let $z_s$ denote the least positive number so that $z_s - \delta(f;s) \in A$ and so that 
$z_s - \delta(f;s) > z_{s'} - \delta(f;s')$ for all $s' < s$.  
Let: 
\begin{eqnarray*}
X(f; 0) & = & \{0\}\\
X(f; s+ 1) & = & X(f;s) \cup \{j 2^{-(\delta(f;s) + z_s)}\ :\ 0 \leq j < 2^{z_s}\}\\
X(f) & = & \bigcup_s X(f;s)\\
L_f & = & (X(f), <)
\end{eqnarray*}
It follows that $L_f$ is an $f$-computably presentable linear order uniformly in $f$.  
Hence, there is a computable operator $F : \N^\N \rightarrow 2^\N$ so that $F(f)$ is the diagram
of a presentation of $L_f$ for all $f \in \N^\N$.

Suppose $P(f)$, and let $x$ be the least number so that for every $y$ there exists $z$ so that 
$Q(f; x,y,z)$.  Thus, $\delta(f;s) = x$ for infinitely many $s$, and $\delta(f;s) < x$ for only finitely many $x$.  It follows that $[0,2^{-x}) \cap X(f)$ is isomorphic to $\Q$ and that $[2^{-x}, 1] \cap X(f)$ is finite.  Thus, $L_f$ is isomorphic to $\eta + n$ for some $n \geq 1$, and the construction ensures $n$ has the correct form.  

On the other hand, suppose $\neg P(f)$.  Then, $X(f) \cap [2^{-x}, 1]$ is finite for all $x \in \N$.  
It follows that $L_f$ is isomorphic to $\omega$. 
\end{proof}

\begin{corollary}\label{cor:sigma.03}
Suppose $P \subseteq \N^\N$ is $\Sigma_3^0$.  Then, there is a computable operator 
$F : \N^\N \rightarrow \N^\N$ so that for all $f \in \N^\N$, there is a Banach space presentation 
$\B_f^\#$ and an $n \in \N$ so that $F(f)$ names $\B_f^\#$ and 
\[
B_f = \left\{ 
\begin{array}{cc}
\ell^p_n \oplus L^p[0,1] & \mbox{if $P(f)$}\\
\ell^p & \mbox{otherwise.}
\end{array}
\right.
\]
Furthermore, if $A \subseteq \N$ is infinite and computable, then 
we can ensure that $n$ has the form $2^{s_1} + \ldots + 2^{s_m}$ where $s_1, \ldots, s_m \in A$ are distinct.
\end{corollary}

\begin{proof}[Proof of Theorem \ref{thm:known.exp}.\ref{thm:known.exp.lp}]
We first show there is a $\Pi_3^0$-predicate $P$ so that whenever $f \in \N^\N$ names an $L^p$ space presentation $\B^\#$, $\ell^p$ isometrically embeds into $\B$ if and only if 
$P(f)$.   To this end, suppose $\psi$ is a disintegration of an $L^p$ space $\B$, and let 
$\{C_n\}_{n < \kappa}$ be a decomposition of $\dom(\psi)$ into almost norm-maximizing chains.  Let $g_n$ denote the $\preceq$-infimum of $\phi[C_n]$ for each $n$.  Then, by Theorem \ref{thm:classification}, $\ell^p$ embeds into $\B$ if and only if $g_n \neq \zerovec$ for infinitely many $n$.  Hence, $\ell^p$ embeds into $\B$ if and only if 
\[
\forall n\ \exists m,k\ \forall \nu \in C_m\ [m > n\ \wedge\ \norm{\psi(\nu)}_\B > 2^{-k}].
\]
The existence of $P$ follows from the technology of Proposition \ref{prop:interpretation}.

It now follows from Lemma \ref{lm:Lp01.embed} that the set of all names of 
presentations of $\ell^p$ is $\Pi_3^0$.  Completeness is obtained from Corollary \ref{cor:sigma.03}.
\end{proof}

\begin{lemma}\label{lm:pi02}
Fix $n \geq 1$.  Suppose $P \subseteq \N^\N$ is $\Pi_2^0$.  Then, there is a computable operator
$F: \N^\N \rightarrow 2^\N$ so that for every $f \in \N^\N$, there is an $n \in \N$ so that $F(f)$ is the diagram of a linear order $L_f$ so that 
\[
L_f \approx \left\{
\begin{array}{cc}
\eta & \mbox{if $P(f)$}\\
\eta + n & \mbox{if $\neg P(f)$}\\
\end{array}
\right.
\]
\end{lemma}

\begin{proof}
Let $Q \subseteq \N^\N \times \N^2$ be a computable predicate so that for all $f \in \N^\N$, 
\[
P(f) \Leftrightarrow \forall x \exists y Q(f; x,y).
\]
Let:
\begin{eqnarray*}
P_1(f; z) & \Leftrightarrow & \forall x < z \exists y\ Q(f; x,y)\\
X_0(f) & = & \bigcup_{z \in P_1(f)} (0,z+1) \cap \Q\\
X_1(f) & = & \bigcup_{z \in P_1(f)} \{z+1, \ldots, z + n\}\\
X(f) & = & X_0(f) \cup X_1(f)\\ %\cup (0,1) \cap \Q \cup \{2,\ldots, 1+n\}\\
L_f & = & (X(f), <)
\end{eqnarray*}
Hence, $L_f$ is an $f$-computably presentable linear order uniformly in $f$.  Thus, there is a 
computable operator $F : \N^\N \rightarrow 2^\N$ so that $F(f)$ is the diagram of $L_f$ for all $f \in \N^\N$.  

If $P(f)$, then $X(f) = X_0(f) = (0, \infty) \cap \Q$.  
If $\neg P(f)$, and if $z_0$ is the least number so that $\neg Q(f; z,y)$ for all $y \in \N$, then 
\[
X(f) = (0, z_0+1) \cap \Q \cup \{z_0 + 1, \ldots, z_0 + n\}.
\]
Therefore, 
\[
L_f \approx \left\{
\begin{array}{cc}
\eta & \mbox{if $P(f)$}\\
\eta + n & \mbox{if $\neg P(f)$}\\
\end{array}
\right.
\]
\end{proof}

\begin{corollary}\label{cor:pi02}
Let $p$ be a computable real so that $p \geq 1$, and let $n \geq 1$.  
Suppose $P \subseteq \N^\N$ is $\Pi_2^0$.  Then, there is a computable operator 
$F : \N^\N \rightarrow \N^\N$ so that for all $f \in \N^\N$, $F(f)$ names a Banach space presentation $\B^\#$ so that 
\[
\B = \left\{
\begin{array}{cc}
L^p[0,1] & \mbox{if $P(f)$}\\
\ell^p_n \oplus L^p[0,1] & \mbox{otherwise}. \\
\end{array}
\right.
\]
\end{corollary}

\begin{proof}
By Lemma \ref{lm:pi02}, there is an operator $G : \N^\N \rightarrow 2^\N$ so that for all $f \in \N^\N$, 
$G(f)$ is the diagram of a linear order $L_f$ such that 
\[
L_f \approx \left\{
\begin{array}{cc}
\eta & \mbox{if $P(f)$}\\
\eta + n + 1 & \mbox{if $\neg P(f)$}\\
\end{array}
\right.
\]
Let $F_0$ be a computable operator as in Proposition \ref{prop:transfer}, and let 
$F = F_0 \circ G$.  

Let $f \in \N^\N$.  Let $L^p(\Omega_f)^\#$ be an $L^p$-space presentation so that $F(f)$ is the diagram of 
$L^p(\Omega_f)^\#$.  Since $L_f$ has an interval that is isomorphic to $\Q$, $L^p[0,1]$ isometrically embeds into $L^p(\Omega)$.  If $P(f)$, then $L_f$ has exactly $n$ adjacencies, and so $\Omega_f$ has exactly $n$ atoms thus ensuring that $L^p(\Omega_f)$ is isometrically isomorphic to 
$\ell^p_n \oplus L^p$.  But, if $\neg P(f)$, then $L_f$ has no adjacencies, and so $L^p(\Omega_f)$ is isometrically isomorphic to $L^p[0,1]$.  
\end{proof}

\begin{proof}[Proof of Theorem \ref{thm:known.exp}.\ref{thm:known.exp.Lp01}]
Fix a computable name $u$ of $p$.

We first show that there is a $\Pi_2^0$ predicate $H \subseteq \N^\N$ so that whenever $f$ names an $L^p$ space presentation $\B^\#$, $H(f)$ if and only if $\B$ is isometrically isomorphic to $L^p[0,1]$.  
To this end, let $J(f,g,h)$ if and only if $f \models \norm{\tau} > 0$ for some $\tau$ and for every $k,n \in \N$, 
if $(f,g,h) \models \nu \in C_n$ for some $\nu \in C_n$, then there is a $\nu$ so that 
$(f,g,h) \models \nu \in C_n$ and $(f,g) \models \norm{\phi(\nu)} < 2^{-k}$.  
Thus, $J$ is $\Pi_2^0$.   Suppose $f$ names $L^p(\Omega)^\#$, $g$ names a disintegration $\psi$ of 
$L^p(\Omega)^\#$, and $h$ names a partition $\{C_n\}_{n < \kappa}$ of $\dom(\psi)$ into almost norm-maximizing chains.  Let $g_n$ denote the $\preceq$-infimum of $\phi[C_n]$.  By Theorem \ref{thm:limits.chains}, $L^p(\Omega)$ is isometrically isomorphic to $L^p[0,1]$ if and only if 
$g_n = \zerovec$ for each $n < \kappa$.  Thus, $J(f,g,h)$ if and only if 
$L^p(\Omega)$ is isometrically isomorphic to $L^p[0,1]$.  
So, we let $H(f)$ if and only if $J(f, \Phidisint(f, u), \Phidecomp(f,\Phidisint(f,u),u))$.  

It now follows from Lemma \ref{lm:Lspace} that the set of all names of presentations of $L^p[0,1]$ is $\Pi_2^0$.  By Lemma \ref{lm:pi02}, it is also $\Pi_2^0$-complete.
\end{proof}

\begin{proof}[Proof of Theorem \ref{thm:known.exp}.\ref{thm:known.exp.lpnLp01}]
Suppose $\psi$ is a disintegration of $L^p(\Omega)$, and suppose $\{C_j\}_{j < \kappa}$ is a decomposition of $S := \dom(\psi)$ into almost norm maximizing chains.  
We claim that $L^p(\Omega)$ is isometrically isomorphic to $\ell^p_n \oplus L^p[0,1]$ if and only if 
the following hold. 
\begin{enumerate}
	\item $\kappa = \omega$.\label{cond1}
	
	\item There exist distinct $j_1, \ldots, j_n \in \N$  and $k \in \N$ so that $\norm{\phi(\nu)}_p > 2^{-k}$ for all $\nu \in C_{j_1} \cup \ldots \cup C_{j_n}$. \label{cond2}
	
	\item For every finite $F \subseteq \N$ and every $k \in \N$, if $\# F > n$, then there exists 
	$\nu \in \bigcup_{j \in F} C_j$ so that $\norm{\psi(\nu)}_p < 2^{-k}$.  \label{cond3}
\end{enumerate}
Suppose $L^p(\Omega)$ is isometrically isomorphic to $\ell^p_n \oplus L^p[0,1]$.  Thus, by Theorem \ref{thm:classification}, $\Omega$ has exactly $n$ atoms.  Let 
$g_j$ denote the $\preceq$-infimum of $\phi[C_j]$ for each $j < \kappa$.  
By Theorem \ref{thm:limits.chains}, there are exactly $n$ values of $j$ so that 
$g_j \neq \zerovec$; label these values $j_1, \ldots, j_n$.  Condition (\ref{cond2}) now follows. 

By way of contradiction, suppose $\kappa < \omega$.  Let $C_j^-$ denote the downset of $C_j$; i.e. $\nu \in C_j^-$ if and only if $\nu \subseteq \nu'$ for some $\nu' \in C_j$.  Then, by Definition \ref{def:anm.chain}, each $C_j^-$ is a branch of $S$.  
Since $\psi$ is injective and summative, each node of $S$ that has a child in $S$ has at least two children in $S$.  
Since $\kappa < \omega$, $S$ has a finite number of branches and so 
$S$ is finite.  Thus, 
each $g_j$ is the $\preceq$-minimal element of $\phi[C_j]$.  
Therefore, since $\psi$ is summative, $\psi(\emptyset) = \sum_{j < \kappa} g_j = \sum_{s = 1}^n g_{j_s}$.  
It follows that $L^p(\Omega)$ is isometrically isomorphic to $\ell^p_n$- a contradiction.  Therefore, 
$\kappa = \omega$.  

Suppose $F \subseteq \N$ is finite and $\# F > n$.  Then, since $\Omega$ has exactly $n$ atoms, by Theorem \ref{thm:limits.chains}, there exists $j \in F$ so that $g_j = \zerovec$.  Hence, (\ref{cond3}).  

Now, suppose (\ref{cond1}) - (\ref{cond3}) hold.  By (\ref{cond2}) and (\ref{cond3}), there are exactly $n$ values of $j$ so that $g_j \neq \zerovec$.  Thus, by Theorem \ref{thm:limits.chains}, $\Omega$ has exactly $n$ atoms.  This means that $L^p(\Omega)$ is isometrically isomorphic to $\ell^p_n$ or $\ell^p_n \oplus_p L^p[0,1]$.  In the former case, $S$ would be finite which is ruled out by (\ref{cond1}).  Hence, 
$L^p(\Omega)$ is isometrically isomorphic to $\ell^p_n \oplus L^p[0,1]$.  

By means of Proposition \ref{prop:interpretation}, it follows that there is a $d$-$\Sigma_2^0$ predicate $P_0$ so that for all $f,g,h \in \N^\N$, if $f$ names 
$L^p(\Omega)^\#$, and if $g$ names a disintegration $\psi$ of $L^p(\Omega)^\#$, and if $h$ names 
a partition $\{C_n\}_{n < \kappa}$ of $\dom(\psi)$ into almost norm-maximizing chains, then 
$P_0(f,g,h)$ if and only if $L^p(\Omega)$, $\psi$, $\{C_n\}_{n < \kappa}$ satisfy (\ref{cond1}) - (\ref{cond3}).  
Let $P_1(f)$ if and only if $P_0(f, \Phidisint(f,r), \Phidecomp(f,\Phidisint(f,r),r))$ and let $P(f)$ if and only if 
$\Lspace(f,r) \wedge P_1(f)$, where $r$ is a computable name of $p$.  Thus, $P$ is $d$-$\Sigma_2^0$ and has all required properties.  

To see completeness, suppose $P$ is a $d$-$\Sigma_2^0$ predicate.  Let $Q$,$R$ be 
$\Sigma_2^0$ predicates so that $P = Q \wedge \neg R$.  We can assume $R \Rightarrow Q$.  
By Corollary \ref{cor:pi02}, for each $X \in \{Q, R\}$, there is a computable operator 
$F_X : \N^\N \rightarrow \N^\N$ so that for all $f \in \N^\N$, $F_X$ names a Banach space presentation 
$\B_f^\#$ so that 
\[
\B_f = \left\{ 
\begin{array}{cc}
\ell^p_n \oplus L^p[0,1] & \mbox{if $X(f)$}\\
L^p[0,1] & \mbox{otherwise}.\\
\end{array}
\right.
\]
Let $G : (\N^\N)^2 \rightarrow \N^\N$ be a computable operator so that whenever $f,g$ name Banach space presentations $\B_0^\#$, $\B_1^\#$, $G(f,g)$ names $(\B_0 \oplus_p \B_1)^\#$.  
Let $F(f) = G(F_Q(f), F_R(f))$.  Thus, $F(f)$ names a presentation of $\ell^p_n \oplus L^p[0,1]$ if and only if $P(f)$.
\end{proof}

\begin{proof}[Proof of Theorem \ref{thm:known.exp}.\ref{thm:known.exp.lpLp01}]
Suppose $\psi$ is a disintegration of $L^p(\Omega)$, and let $\{C_n\}_{n < \kappa}$ be a partition of $S: = \dom(\psi)$ into almost norm-maximizing chains.  

We begin by showing that for every $n_0 < \kappa$, there exist pairwise incomparable 
$\nu_0, \ldots, \nu_{n_0}$ so that $\nu_j \in C_j$ for each $j$.  For, let 
$m_j = \min\{|\nu|\ :\ \nu \in C_j\}$ for each $j$, and let $M = \max\{m_0, \ldots, m_{n_0}\}$.  
Hence, if $C_j$ does not have a node of length $M$, then $C_j$ is bounded.
Thus, by Definition \ref{def:anm.chain}, if $C_j$ does not have a node of length $M$, then 
$C_j$ contains a terminal node of $S$.  For each $j$ so that $C_j$ contains a node of length $M$, 
let $\nu_j$ be such a node.  For all other $j \leq n_0$, let $\nu_j$ denote the terminal node of $C_j$. 
Thus, $|\nu_j| \leq M$ for each $j \leq n_0$, and $\nu_j$ is terminal if $|\nu_j| < M$.  
Therefore, $\nu_0, \ldots, \nu_{n_0}$ are incomparable.

Let $g_n$ denote the $\preceq$-infimum of $\psi[C_n]$.

We claim $L^p(\Omega)$ is isometrically 
isomorphic to $\ell^p \oplus L^p[0,1]$ if and only if both of the following hold.
\begin{enumerate}
	\item For every $m \in \N$, there exist $n,k \in \N$ with $m < n < \kappa$ such that for all 
	$\nu \in C_n$  $\norm{\psi(\nu)}_p > 2^{-k}$.  \label{thm:lp.plus.Lp01::cond.1}
	
	\item There exists $k \in \N$ so that for all $m < \kappa$ there exists pairwise incomparable $\nu_0, \ldots, \nu_m$ so that 
$\nu_n \in C_n$ for each $n \leq m$ and $\norm{\psi(\emptyset) - \sum_{n \leq m} g_n}_p > 2^{-k}$.  \label{thm:lp.plus.Lp01::cond.2}
\end{enumerate}

For, suppose $L^p(\Omega)$ is isometrically isomorphic to $\ell^p \oplus L^p[0,1]$.  
Thus, $\Omega$ has $\aleph_0$ atoms.  So, by Theorem \ref{thm:limits.chains}, there are 
infinitely many $n$ so that $g_n \neq \zerovec$.  Since 
$g_n \preceq \phi(\nu)$ for all $\nu \in C_n$, (\ref{thm:lp.plus.Lp01::cond.1}) follows.  It also follows that 
$\kappa = \omega$.  

Since $L^p(\Omega)$ is not isometrically isomorphic to $\ell^p$, 
$\psi(\emptyset) \neq \sum_n g_n$.  Choose $k \in \N$ so that 
$2^{-k} < \norm{\psi(\emptyset) - \sum_n g_n}_p$.  Since $g_0, g_1, \ldots$ are disjointly supported
components of $\psi(\emptyset)$, $\norm{\psi(\emptyset) - \sum_n g_n}_p^p = 
\norm{\psi(\emptyset)}_p^p - \sum_n \norm{g_n}_p^p$.  Therefore, 
$\norm{\psi(\emptyset)}_p^p - \sum_{n = 0}^m \norm{g_n}_p^p > 2^{-kp}$ for all $m \in \N$.  
So, let $m \in \N$.  Then by Proposition \ref{prop:subvectorLimitsExist} there exist $\nu_0, \ldots, \nu_m$ so that $\nu_n \in C_n$ for each $n \leq m$ and so that 
$\norm{\psi(\emptyset)}_p^p - \sum_{n = 0}^m \norm{\psi(\nu_n)}_p^p > 2^{-kp}$.  Hence, 
$\norm{\psi(\emptyset) - \sum_{n = 0}^m \psi(\nu_n)}_p > 2^{-k}$.  
Thus, (\ref{thm:lp.plus.Lp01::cond.2}).

Conversely, suppose (\ref{thm:lp.plus.Lp01::cond.1}) and (\ref{thm:lp.plus.Lp01::cond.2}) hold.  
Hence, $\kappa = \omega$, and there are infinitely many values of $n$ so that 
$g_n \neq \zerovec$.  So by Theorem \ref{thm:limits.chains}, $\Omega$ has $\aleph_0$ atoms. 
Thus, by Theorem \ref{thm:classification}, $L^p(\Omega)$ is either $\ell^p$ or $\ell^p \oplus L^p[0,1]$.  
Let $k \in \N$ be as given by (\ref{thm:lp.plus.Lp01::cond.2}).  We claim that 
$\norm{\psi(\emptyset) - \sum_{n \leq m} g_n}_p > 2^{-k}$ for all $m \in \N$.  
For, let $m \in \N$.  Then, there exist pairwise incomparable $\nu_0, \ldots, \nu_m$ so that 
$\nu_j \in C_j$ for each $j$ and $\norm{\psi(\emptyset) - \sum_{n = 0}^m \psi(\nu_n)}_p > 2^{-k}$.  
Thus, $\norm{\psi(\emptyset)}_p^p - \sum_{n = 0}^m \norm{\psi(\nu_n)}_p^p > 2^{-kp}$.  
Therefore, for all $\nu_0', \ldots, \nu_m'$, if $\nu_j \subseteq \nu_j' \in C_j$ for each $j$, then 
$\norm{\psi(\emptyset)}_p^p - \sum_{n = 0}^m \norm{\psi(\nu_n')}_p^p > 2^{-kp}$.  
Hence, $\norm{\psi(\emptyset)}_p^p - \sum_{n \leq m} \norm{g_n}_p^p > 2^{-kp}$, and so 
$\norm{\psi(\emptyset) - \sum_{n \leq m} g_n}_p > 2^{-k}$.

Thus, $\psi(\emptyset) \neq \sum_n g_n$.  Hence, the support of 
$\psi(\emptyset) - \sum_n g_n$ has positive measure and includes no atom of $\Omega$.  
Therefore, $\Omega$ is not purely atomic, and so $L^p[0,1]$ isometrically embeds into 
$L^p(\Omega)$.  Thus, by Theorem \ref{thm:classification}, $L^p(\Omega)$ is isometrically isomorphic to $\ell^p \oplus L^p[0,1]$. 

To demonstrate completeness, let $P \subseteq \N^\N$ be a $d$-$\Sigma_3^0$ predicate. 
Let $Q$,$R$ be $\Sigma_3^0$ predicates so that $P = Q \wedge \neg R$.  Without loss of generality, 
assume $R \Rightarrow Q$.  By Corollary \ref{cor:sigma.03}, for each $X \in \{Q,R\}$, there is a computable operator $F_X : \N^\N \rightarrow \N^\N$ so that for every $f \in \N^\N$ there is an $n\in \N$ and a Banach space presentation $\B_f^\#$ so that $F(f)$ names $\B_f^\#$ and 
\[
\B_f = \left\{\begin{array}{cc}
\ell^p_n \oplus L^p[0,1] & \mbox{if $X(f)$} \\
\ell^p& \mbox{otherwise.}\\
\end{array}
\right.
\]
There is a computable operator $H: (\N^\N)^2 \rightarrow \N^\N$ so that 
whenever $f,g$ name Banach space presentations, $H(f,g)$ names the $L^p$-sum of these presentations.  Let $F(f) = H(F_Q(f), F_R(f))$.  If $P(f)$, then $F(f)$ names 
a presentation of $\ell^p \oplus L^p[0,1]$.  If $R(f)$, then $F(f)$ names a presentation of 
$\ell^p_n \oplus L^p[0,1]$ for some $n \in \N$.  If $\neg Q(f)$, then $F(f)$ names a presentation 
of $\ell^p$.  Hence, $P(f)$ if and only if $F(f)$ names a presentation of 
$\ell^p \oplus L^p[0,1]$.
\end{proof}

\section{The index set of all computable Lebesgue space presentations}\label{sec:index.Lebesgue}

We now prove Main Theorem \ref{mthm:Lebesgue} which we now restate in name form as follows.

\begin{theorem}\label{thm:pres.Lebesgue}
The set of all names of Lebesgue space presentations is $\Pi_3^0$.
\end{theorem}

The proof of Theorem \ref{thm:pres.Lebesgue} comes down to the following.

\begin{lemma}\label{lm:comp.exp}
There exist computable operators $F_0 : \N^\N \rightarrow \Q^\N$ and $F_1 : \N^\N \rightarrow \Q^\N$ so that whenever $f \in \N^\N$ names a presentation of a Lebesgue space $\B$ of dimension at least 2, 
$F_0(f)$ enumerates the right Dedekind cut of the exponent of $\B$ if this exponent is at most $2$, and  $F_1(f)$ enumerates the left Dedekind cut of the exponent of $\B$ if this exponent is at least $2$. 
\end{lemma}

\begin{proof}
Let $Y(f)$ consist of all $r \in \Q \cap(1, \infty)$ for which there exists a positive rational number $r_0$
and terms $\tau_0$, $\tau_1$ of $\Lpres$ so that the following hold.
\begin{enumerate}
	\item $f \models  1 < \norm{\tau_j} < 1 + r_0$
	
	\item $f \models \norm{\tau_0 - \tau_1} > 1 + 2r_0$.
	
	\item $f \models r' + 2r_0 < \norm{\tau_0 + \tau_1}$ for some rational number $r'$ so that 
	$r' > 2(1 - \delta(r,1))$.  
\end{enumerate}

Let:
\begin{eqnarray*}
X_0(f) & = & Y(f) \cap (1,2) \cup (2, \infty) \cap \Q\\
X_1(f) & = & Y(f) \cap (2, \infty) \cup (1, \infty) \cap \Q.
\end{eqnarray*}
Since the function $\delta$ is computable, $X_j(f)$ is $f$-c.e. uniformly in $f$.  Thus, 
there is a computable operator $F_j : \N^\N \rightarrow \Q^\N$ so that $F_j(f)$ enumerates 
$X_j(f)$ for each $f$.

Let $p \leq 2$.  Suppose $f$ names $L^p(\Omega)^\#$ and $\dim(L^p(\Omega)) \geq 2$.  
We show $X_0(f)$ is the right Dedekind cut of $p$.  To begin, suppose $r \in X_0(f)$.
Without loss of generality, suppose $r \leq 2$.  Then, by the definition of $Y(f)$, there exists a positive number $r_0$ and rational
vectors $u_0$, $u_1$ of $L^p(\Omega)^\#$ so that 
$1 < \norm{u_j}_p < 1+ r_0$, $\norm{u_0 - u_1}_p >1 + 2r_0$, and 
$2(1 - \delta(r,1)) < \norm{u_0 + u_1} - 2r_0$.  Since $\norm{u_j}_p - r_0 < 1$, 
$B(u_j;r_0)$ contains a unit vector $v_j$.  Thus, $1 < \norm{u_0 - u_1}_p - 2r_0 < \norm{v_0 - v_1}_p$. 
Also, $\norm{u_0 + u_1}_p - 2r_0 < \norm{v_0 + v_1}_p$.  Hence, 
$1 - 2^{-1}\norm{v_0 + v_1}_p < \delta(r,1)$.  Thus, by Proposition \ref{prop:hanner.extended}, $\delta(p,1) < \delta(r,1)$.  So, by Proposition \ref{prop:delta.inc.dec}, $p < r$.  

Conversely, suppose $r$ is a rational number and $p < r$.  We show that $r \in X_0(f)$.  
Without loss of generality, suppose $r \leq 2$.  Then, by Proposition \ref{prop:delta.inc.dec}, 
$\delta(p,1) < \delta(r,1)$.  So, by Proposition \ref{prop:hanner.extended}, there exist unit vectors 
$v_0$, $v_1$ of $L^p(\Omega)$ so that $\norm{v_0 - v_1}_p > 1$ and 
$1 - 2^{-1}\norm{v_0 + v_1}_p < \delta(r,1)$.  There is a positive rational number $r_0$ so that 
for all $u_0 \in B(v_0; r_0)$ and all $u_1 \in B(v_1, r_0)$, 
$\norm{u_0 - u_1}_p - 2r_0 > 1$ and 
$2(1 - \delta(r,1)) < \norm{u_0 + u_1}_p - 2r_0$.  Since 
$B(v_j; r_0) - \overline{B(\zerovec; 1)}$ is open, there exist rational $u_j \in B(v_j; r_0)$ so that 
$1 < \norm{u_j}_p < 1 + r_j$.  It follows that $r \in X_0(f)$.  

The case $p \geq 2$ is similar.
\end{proof}

\begin{proof}[Proof of Theorem \ref{thm:pres.Lebesgue}]
By Lemma \ref{lm:comp.exp}, there exist $\Delta_2^0$-operators $F_0 : \N^\N \rightarrow \N^\N$ and 
$F_1 : \N^\N \rightarrow \N^\N$ so that for all $f \in \N^\N$, each $F_j(f)$ names a real in 
$[1,\infty)$, and if $f$ names a presentation of a Lebesgue space, then one of $F_0(f)$, $F_1(f)$ names the exponent of that space.  For all $f \in \N^\N$, let 
\[
P(f) \Leftrightarrow (\Lspace(f, F_0(f))\ \vee\ \Lspace(f, F_1(f))).
\]
By Lemma \ref{lm:Lspace}, $P$ is $\Pi_3^0$ and consists precisely of the names of Lebesgue space presentations.  
\end{proof}

\section{Isometric isomorphism results}\label{sec:isom}

We now turn to the complexity of the isometric isomorphism problem for $L^p$ spaces.  
Namely, we prove Main Theorem \ref{mthm:isom} by proving the following.

\begin{theorem}\label{thm:isom.fixed.exp}
Let $p$ be a computable real so that $p \geq 1$ and so that $p \neq 2$.  Then, the set of all 
$(f,g) \in \N^\N \times \N^\N$ so that $f,g$ name presentations of isometrically isomorphic
$L^p$ spaces is co-$3$-$\Sigma_3^0$ complete.
\end{theorem}

This proof is accomplished via the following two lemmas.

\begin{lemma}\label{lm:lwr.known.p}
Let $p \geq 1$ be a computable real so that $p \neq 2$.  Suppose $P \subseteq \N^\N$ is 
$3$-$\Sigma_3^0$.  Then, there exist computable operators $F_0, F_1 : \N^\N \rightarrow \N^\N$ so that for every $f \in \N^\N$, each $F_j(f)$ names an $L^p$ space presentation, and these $L^p$ spaces are isometrically isomorphic if and only if $\neg P$.
\end{lemma}

\begin{proof}
The exist $\Sigma_3^0$ predicates $Q_0, Q_1, Q_2 \subseteq \N^\N$ so that 
$P = (Q_0 \wedge \neg Q_1)\ \vee\ Q_2$ and $Q_2 \Rightarrow Q_1 \Rightarrow Q_0$.  
By Corollary \ref{cor:sigma.03}, for each $j$ there is a computable operator $G_j : \N^\N \rightarrow \N^\N$ so that for every $f \in \N^\N$, 
$G_j(f)$ is name a presentation of an $L^p$ space $\mathcal{A}_{f,j}$ so that there is an $n_{f,j} \in \N$ for which 
\[
\mathcal{A}_{f,j} = \left\{
\begin{array}{cc}
\ell^p_{n_{f,j}} \oplus L^p[0,1] & Q_j(f)\\
\ell^p & \neg Q_j(f)\\
\end{array}
\right.
\] 
These operators can be chosen so that if $Q_0(f)$ and $Q_1(f)$, then $n_{f,0} \neq n_{f,1}$.  
Let $H : (\N^\N)^2 \rightarrow \N^\N$ be a computable operator so that for all $f,g \in \N^\N$, if 
$f$ and $g$ name Banach space presentations, then $H(f,g)$ names a presentation of the $L^p$ sum of these Banach spaces.  Finally, let $F_0(f) = H(G_0(f), G_2(f))$ and $F_1(f) = H(G_1(f), G_2(f))$. 

Let $\B_{f,0} = \mathcal{A}_{f,0} \oplus_{L^p} \mathcal{A}_{f,2}$, and let 
$\B_{f,1} = \mathcal{A}_{f,1} \oplus_{L^p} \mathcal{A}_{f,2}$.  Thus, 
$F_j(f)$ names a presentation of $\B_{f,j}$.

Suppose $P(f)$.  We first consider the case where $Q_0(f) \wedge \neg Q_1(f)$ holds. 
Thus, $\neg Q_2(f)$, and so $\mathcal{A}_{f,2} = \ell^p$.  It follows that $\mathcal{A}_{f,0} = \ell_{n_{f,0}}^p \oplus L^p[0,1]$ and $\mathcal{A}_{f,1} = \ell_{n_{f,1}}^p \oplus L^p[0,1]$.  Thus, $\B_{f,0}$ is isometrically 
isomorphic to $\ell^p \oplus L^p[0,1]$ and $\B_{f,1}$ is isometrically isomorphic to $\ell^p$.  
Since $p \neq 2$, these spaces are not isometrically isomorphic to each other.

Now, we consider the case where $Q_2(f)$.  Then, $Q_0(f)$ and $Q_1(f)$.  Let 
$m_j = n_{f,j} + n_{f,2}$.  Thus, $m_0 \neq m_1$.  Also, $\B_{f,j}$ is isometrically isomorphic to 
$\ell^p_{m_j} \oplus L^p[0,1]$.  Thus, since $p \neq 2$, $\B_{f,0}$ and $\B_{f,1}$ are not isometrically isomorphic.  

Now, suppose $\neq P(f)$.  Thus, $\neg Q_2(f)$ and so $\mathcal{A}_{f,2} = \ell^p$.  
Suppose $Q_0(f)$.  Then, $Q_1(f)$.  Thus, for each $j \in \{0,1\}$, $\mathcal{A}_{f,j} = \ell_{n_{f,j}}^p \oplus L^p[0,1]$.  Hence, $\B_{f,0}$ and $\B_{f,1}$ are isometrically isomorphic to $\ell^p \oplus L^p[0,1]$.  
On the other hand, suppose $\neg Q_0(f)$.  Then, $\neg Q_1(f)$.  Therefore, 
$\mathcal{A}_{f,j} = \ell^p$.  Hence, $\B_{f,0}$ and $\B_{f,1}$ are isometrically isomorphic to $\ell^p$.
\end{proof}

\begin{lemma}\label{lm:k.atoms}
Let $p$ be a computable real so that $p \geq 1$ and so that $p \neq 2$.  There is a $\Sigma_2^0$ predicate $P \subseteq \N^\N \times \N$ so that for all $f \in \N^\N$ and all $k \in\N$, if $f$ names a presentation $L^p(\Omega)^\#$, then $P(f;k)$ if and only if $\Omega$ has at least $k$ atoms.
\end{lemma}

\begin{proof}
The lower bound follows from Lemma \ref{lm:lwr.known.p}.  
Suppose $\psi$ is a disintegration of $L^p(\Omega)$, and suppose $\{C_n\}_{n < \kappa}$ is a 
partition of $\dom(\psi)$ into almost norm-maximizing chains.  It follows from Theorem \ref{thm:limits.chains} that $\Omega$ has at least $k$ atoms if and only if there exist distinct 
$n_1, \ldots, n_k < \kappa$ so that $\inf\{\norm{\phi(\nu)}_p\ :\ \nu \in \bigcup_j C_{n_j}\} > 0$.
By means of the technology described in Proposition \ref{prop:interpretation}, it follows that there is 
a $\Sigma_2^0$ predicate $P_0$ so that for all $f,g,h \in \N^\N$, if $f$ names a presentation $L^p(\Omega)^\#$, and if $g$ names a disintegration $\psi$ of $L^p(\Omega)^\#$, and if 
$h$ names a partition $\{C_n\}_{n < \kappa}$ of $\dom(\psi)$ into almost norm-maximizing chains, 
then $P_0(f,g,h)$ if and only if $\Omega$ has at least $k$ atoms.  
The existence of the predicate $P$ now follows from the uniformity of Theorems 
\ref{thm:disint} and \ref{thm:decomp}.
\end{proof}

\begin{proof}[Proof of Theorem \ref{thm:isom.fixed.exp}]
Suppose $\Omega_0$, $\Omega_1$ are separable measure spaces.  It follows from Theorem 
\ref{thm:classification} that $L^p(\Omega_0)$ and $L^p(\Omega_1)$ are isometrically isomorphic
if and only if both of the following hold.
\begin{enumerate}
	\item $L^p[0,1]$ isometrically embeds into both $L^p(\Omega_0)$ and $L^p(\Omega_1)$ or into neither of these spaces.  
	
	\item $\Omega_0$ and $\Omega_1$ have the same number of atoms.  
\end{enumerate}
It follows from Lemma \ref{lm:Lp01.embed} that there is a co-$2$-$\Sigma_3^0$ predicate $Q \subseteq \N^\N \times \N^\N$ so that for all $f,g \in \N^\N$, if $f$ and $g$ name presentations of $L^p$ spaces, then 
$Q(f,g)$ if and only if $L^p[0,1]$ isometrically embeds into both of these spaces or into neither of these spaces.  

Let $R_1 \subseteq \N^\N \times \N$ be a $\Sigma_2^0$ predicate as given by Lemma \ref{lm:k.atoms}.  
Let $R(f,g;k)$. hold if and only if for all $k \in \N$ 
\[
(R_1(f;k) \wedge R_1(g;k)) \vee (\neg R_1(f;k) \wedge \neg R_1(g;k)).
\]
Thus, $R$ is $\Pi_3^0$.  If $f,g$ name $L^p(\Omega_0)^\#$ and $L^p(\Omega_1)^\#$ respectively, then $R(f,g)$ if and only if $\Omega_0$ and $\Omega_1$ have the same number of atoms.  

Let $h$ be a computable name of $p$, and let $L(f,g)$ hold if and only if 
$\Lspace(f,h)$ and $\Lspace(g,h)$.  Thus, by Lemma \ref{lm:Lspace}, $L$ is $\Pi_2^0$.  

Finally, let $P = Q \wedge R \wedge L$.  Thus, $\neg P$ is 3-$\Sigma_3^0$, and 
$P(f,g)$ holds if and only if $f$ and $g$ name presentations of isometrically isomorphic $L^p$ spaces.
\end{proof}

\section{Conclusion}\label{sec:conc}

Our goal has been to use computability theory to gauge the complexity of 
classifying various kinds of Lebesgue spaces and associated isometric isomorphism problems.  
Our results have placed the complexity of these problems in the arithmetical hierarchy and the relativized Ershov hierarchy, and for the most part our analysis has been exact.  For reasons discussed in the introduction, we leave open whether the bound in Main Theorem \ref{mthm:Lebesgue} is sharp, and we believe resolution of this question will require a significant advance in the technology available for building Banach space presentations.  One of our contributions to this technology is a computable functor that, roughly speaking, transforms a linear order $L$ into an $L^p$ space $B(L)$ in such a way that a significant amount of information about a linear order is reflected in the structure of the corresponding $L^p$ space. Our analysis of the functor itself was not optimal; we limited ourselves only to establishing those properties sufficient to prove our theorems. A more detailed study of the functor is left as an open problem.

%\bibliographystyle{amsplain}
%\bibliography{ourbib}
\def\cprime{$'$} \def\cprime{$'$} \def\cprime{$'$} \def\cprime{$'$}
\providecommand{\bysame}{\leavevmode\hbox to3em{\hrulefill}\thinspace}
\providecommand{\MR}{\relax\ifhmode\unskip\space\fi MR }
\providecommand{\MRhref}[2]{%
  \href{http://www.ams.org/mathscinet-getitem?mr=#1}{#2}
}
\providecommand{\href}[2]{#2}

\end{document}